\documentclass[12pt]{article}
\usepackage{amsthm}
\usepackage{amsfonts}
\usepackage{epsfig}
\newtheorem{theorem}{Theorem}
\newtheorem{conjecture}{Conjecture}
\newtheorem{corollary}[theorem]{Corollary}
\newtheorem{lemma}[theorem]{Lemma}

\newcounter{claims}
\newenvironment{claims}{\refstepcounter{claims}\par\medskip\noindent%
{{\bf (\theclaims)}~~}}{\par\medskip}

\newcommand{\ins}{\mbox{ins}}
\newcommand{\faces}{\mbox{mf}}
\newcommand{\bd}{\mbox{bd}}
\newcommand{\TT}{{\cal T}}
\newcommand{\DD}{{\cal D}}
\newcommand{\LL}{{\cal L}}
\newcommand{\SSS}{{\cal S}}

\title{A stronger structure theorem for excluded topological minors}

\author{Zden\v{e}k Dvo\v{r}\'ak\thanks{Computer Science Institute of Charles University, Prague, Czech Republic.
E-mail: {\tt rakdver@iuuk.mff.cuni.cz}.
The work leading to this invention has received funding from the European
Research Council under the European Union's Seventh Framework Programme
(FP7/2007-2013)/ERC grant agreement no. 259385.}}
\date{}

\begin{document}
\maketitle

\begin{abstract}
Grohe and Marx~\cite{gmarx} proved that if $G$ does not contain $H$ as a topological minor,
then there exist constants $g=O(|V(H)|^4)$, $D$ and $t$ depending only on $H$ such that $G$ is a clique sum of graphs that
either contain at most $t$ vertices of degree greater than $D$ or almost embed in some surface of genus at most $g$.
We strengthen this result, giving a more precise description of the latter kind of basic graphs of
the decomposition---we only allow graphs that (almost) embed in ways that are impossible for $H$
(similarly to the structure theorem for minors, where only graphs almost embedded in surfaces
in that $H$ does not embed are allowed).
This enables us to give structural results for graphs avoiding a fixed graph as an immersion
and for graphs with bounded $\infty$-admissibility.
\end{abstract}

In their fundamental result, Robertson and Seymour~\cite{robertson2003graph} gave a structural characterization
of graphs avoiding a fixed minor, showing that they can be decomposed along small cuts to
graphs that almost embed in some surface of bounded genus.
\begin{theorem}\label{thm-robsey}
For every graph $H$, there exists a constant $k$ such that if $G$ does not contain $H$ as a minor, then
$G$ is a clique-sum of graphs that $k$-almost-embed in surfaces in that $H$ cannot be embedded.
\end{theorem}
Here, a \emph{clique-sum} of graphs $G_1$ and $G_2$ is a graph obtained from them by identifying vertices of a clique in $G_1$
with vertices of a clique in $G_2$ of the same size, and possibly removing some edges.  A graph $G$ \emph{$k$-almost-embeds} in
a surface $\Sigma$ if a subgraph of $G$ obtained by removing at most $k$ vertices is an outgrowth of a graph embeddable in $\Sigma$
by at most $k$ vortices of depth at most $k$.  The definition of an outgrowth by vortices is somewhat technical and
we postpone it for later; at the moment, it is sufficient to know that each vortex is a restricted graph
attaching to a boundary of a single face of the embedding.

This result had a tremendous impact on structural graph theory and enabled development of many efficient
algorithms for graphs avoiding a fixed minor.  A corollary of the theory was an algorithm~\cite{rs13} to test
whether a graph $G$ on $n$ vertices contains a fixed graph $H$ as a minor in $O(n^3)$ time (only the multiplicative
constant hidden in the $O$-notation depends on $H$).

The methods of~\cite{rs13} also make it possible to test whether $H$ appears as a topological minor in $G$ (i.e., whether a subdivision
of $H$ is a subgraph of $G$) in $O(n^{|V(H)|+3})$.  For a long time, it was open whether topological minor containment of a fixed graph
$H$ can be tested in time $O(n^c)$ for some constant $c$ independent on $H$.  Finally, Grohe et al.~\cite{topmintest} gave a positive
answer to this question.  Further refining the ideas of this algorithm, Grohe and Marx~\cite{gmarx} were able to derive a structural
characterization of graphs avoiding a topological minor.
\begin{theorem}\label{thm-gmarx}
For every graph $H$, there exist constants $g=O(|V(H)|^4)$, $t$, $D$ and $k$ such that if $G$ does not contain $H$ as a topological
minor, then $G$ is a clique-sum of graphs that either contain at most $t$ vertices of degree greater than $D$,
or $k$-almost-embed in surfaces of genus at most $g$.
\end{theorem}
They also gave examples using this characterization to design algorithms for graphs avoiding a fixed topological minor,
most importantly providing a polynomial-time isomorphism test for such graphs.

Comparing Theorems~\ref{thm-robsey} and \ref{thm-gmarx}, the latter places much weaker restriction on the almost embedded
pieces of $G$.  It would be desirable to replace ``surfaces of genus at most $g$'' by ``surfaces in that $H$ does not embed''.
Unfortunately, the resulting claim is false.  Consider the double-wheel $W_n$
with $n\ge 5$, consisting of a cycle $C$ of length $n$ and two vertices adjacent to $V(C)$.
The graph $W_n$ is planar, and thus there is no surface in that $W_n$ does not embed.  Should the proposed claim be true,
we would have that every graph that does not contain $W_n$ as a topological minor is a clique-sum of graphs
with at most $t$ vertices of degree greater than $D$, for some constants $t$ and $D$.
However, note that $W_n$ has a unique embedding in the sphere and the two vertices of degree
$n$ are not incident with a common face in this embedding.
Consequently, if $G$ is any graph embedded in the sphere such that all vertices of degree at least $n$ are incident with one face, then $G$ cannot contain $W_n$ as
a topological minor.  It is easy to see that the graphs embedded in the sphere with all vertices of degree at least $n$
incident with one face do not admit the proposed decomposition.

Therefore, the strengthening of Theorem~\ref{thm-gmarx} that we develop in this paper
needs to take properties of embeddings of $H$ in particular surfaces into account.
Let $\Sigma$ be a surface without boundary and let $H$ be a graph.
A set $S$ of faces in an embedding of $H$ is \emph{dominating $(\ge\!4)$-vertices} if
every vertex of $H$ of degree at least four is incident with a face in $S$.
If $H$ can be embedded in $\Sigma$, then let $\faces(H,\Sigma)$ be the minimum size of a set dominating $(\ge\!4)$-vertices,
taken over all embeddings of $H$ in $\Sigma$.  If $H$ cannot be embedded in $\Sigma$, then let $\faces(H,\Sigma)=\infty$.

In our version of Theorem~\ref{thm-gmarx}, we essentially
require that if a piece $G'$ of the decomposition almost embeds in a surface $\Sigma$ in that $H$ also embeds,
then all vertices of $G'$ of degree greater than $D$ are close to one of at most $\faces(H,\Sigma)-1$
faces of $G'$.  An equivalent and a slightly more convenient formulation is that all vertices of degree greater than $D$
of $G'$ are contained in one of at most $\faces(H,\Sigma)-1$ vortices.

However, this is still not quite right---even if $G'$ contains many high-degree vertices spread throughout it,
it may not contain $H$ as a topological minor if these high-degree vertices are separated from each other
(or from non-planar parts of $G'$, when $H$ is non-planar) by small separators.  These separators
cannot be eliminated through clique-sums, as there can be arbitrarily many of them and replacing them
by cliques would destroy the embedding of $G'$.  Therefore, we need one more (rather technical) adjustment for the
basic pieces of the decomposition, that allows for such separated vertices of high degree.  Let us now introduce
several definitions in order to formulate these ideas precisely.

A \emph{surface} is a compact connected $2$-manifold, possibly with a boundary.
In an embedding of a graph $G$ in a surface $\Sigma$ with boundary $\bd(\Sigma)$, we require that
$G$ intersects $\bd(\Sigma)$ only in vertices.  A \emph{face} of $G$ is an arc-connected component of
$\Sigma\setminus G$.  A simple curve in $\Sigma$ is \emph{$G$-normal} if it only intersects
$G$ in vertices.  Let $\Sigma(c,h,b)$ denote the surface obtained from the sphere with $b$ holes by adding $c$
crosscaps and $h$ handles.  By the classification theorem, every surface is homeomorphic to $\Sigma(c,h,b)$
for some $c,h,b\ge 0$.  If $\Sigma=\Sigma(c,h,b)$, then let $\hat{\Sigma}$ denote $\Sigma(c,h,0)$.
The \emph{genus} of $\Sigma(c,h,b)$ is $c+2h$.  The surface $\Sigma(0,0,1)$ is called a \emph{(closed) disk}.
The interior of $\Sigma(0,0,1)$ is an \emph{open disk}.

Let $G_1$ be a graph embedded in a surface $\Sigma_1$ and let $G_2$ be embedded in a disk $\Sigma_2$.
Let $u_1$, $u_2$, \ldots, $u_k$ be the vertices of $G_2\cap \bd(\Sigma_2)$ in the order
along the curve $\bd(\Sigma_2)$.  Let $f$ be a face of $G_1$ homeomorphic to an open disk $\Lambda$ bounded by a cycle $C=v_1v_2\ldots v_k$.
Let $G$ be the graph embedded in $\Sigma_1$ obtained from $G_1$ by adding a homeomorphic copy of $G_2$ to the closure of $\Lambda$
so that $u_i$ is mapped to $v_i$ for $1\le i\le k$.  We say that $G$ is obtained from $G_1$ by \emph{pasting} $G_2$.

Let $t$ be either a nonnegative integer or $\infty$ and let $D$, $n$, $m$ and $a$ be nonnegative integers.  A graph $G$ embedded in a disk $\Sigma$
is \emph{$(n,t,D,m,a)$-basic} if $|G\cap\bd(\Sigma)|\le n$ and there exist sets $S\subseteq V(G)$
and $A\subseteq V(G)$ with $|S|<t$ and $|A|\le a$ such that for every vertex $v\in V(G)\setminus A$ of degree greater than $D$,
there exist $s\in S$ and a $G$-normal
simple curve in $\Sigma$ joining $v$ with $s$ and intersecting $G$ in at most $m$ vertices.  That is, high degree vertices
(up to at most $a$ exceptions) of $G$ are contained in less than $t$ subgraphs near to the vertices of $S$.   Note that the condition
is vacuous when $t=\infty$.  A graph $G$ is an \emph{$(n,t,D,m,a)$-patch} if it is obtained from $(n,t,D,m,a)$-basic
graphs by a sequence of pastings.  Let us remark that $G$ is an $(n,\infty,D,m,a)$-patch iff it is embedded in a disk with at most $n$ vertices
of $G$ in its boundary.

Let $G=G_0\cup G_1\cup \ldots\cup G_k$, where $G_0$ is edge-disjoint with $G_i$ for $1\le i\le k$, and
$G_i$ is vertex-disjoint with $G_j$ for $1\le i<j\le k$.  Suppose that
$G_0$ is embedded in a surface $\Sigma=\Sigma(c,h,k)$, let $B_1$, \ldots, $B_k$ be the components of $\bd(\Sigma)$
and assume that $V(G_0)\cap V(G_i)=G_0\cap B_i$ for $1\le i\le k$.  
Let the vertices of $G_0\cap B_i$ in order along $B_i$ be $v^i_1$, \ldots, $v^i_{r_i}$.
Furthermore, suppose that for $1\le i\le k$, there exists a path decomposition $R^i_1$, \ldots, $R^i_{r_i}$ of $G_i$ of width at most $p$
such that $v^i_j$ belongs to $R^i_j$ for $1\le j\le r_i$.  In this situation, we say that $G$ is an \emph{outgrowth
of $G_0$ by $k$ vortices of depth at most $p$} and we call the graphs $G_1$, \ldots, $G_k$ \emph{vortices}.
A graph $G'$ is an \emph{$(n,t,D,m, a)$-expansion} of $G$ if $G'=G'_0\cup G_1\cup\ldots\cup G_k$ and $G'_0$ is a subgraph of
a graph obtained from $G_0$ by pasting $(n,t,D,m,a)$-patches to distinct faces of $G_0$.

Finally, we are ready to formulate our main result.

\begin{theorem}\label{thm-main}
For every $H$, there exist constants $n$, $D$, $m$, $k$, $p$ and $a$ with the following property.
Every graph $G$ that does not contain $H$ as a topological minor can be expressed as a clique-sum of
graphs $G_1$, $G_2$, \ldots, $G_l$ such that for $1\le i\le l$, there exists a set $A_i\subset V(G_i)$ of size at most $a$
and $G'_i=G_i-A_i$ satisfies one of the following conditions:
\begin{itemize}
\item the maximum degree of $G'_i$ is at most $D$, or
\item for some surface $\Sigma$ in that $H$ cannot be embedded,
$G'_i$ is an outgrowth of a graph embedded in $\Sigma$ by at most $k$ vortices of depth at most $p$, or
\item for some surface $\Sigma$ in that $H$ can be embedded, satisfying $\faces(H,\Sigma)\ge 2$,
there exists an outgrowth $G''_i$ of a graph embedded in $\Sigma$ by at most $k$ vortices of depth at most $p$ such that
all vertices $v\in V(G''_i)$ with $\deg_{G''_i}(v)>D$ belong to the vortices,
less than $\faces(H,\hat{\Sigma})$ of the vortices contain such a vertex,
and $G'_i$ is an $(n,\faces(H,\Sigma(0,0,0)),D,m,a)$-expansion of $G''_i$.
\end{itemize}
\end{theorem}

The proof of this theorem is presented in Section~\ref{sec-struct}.
Let us remark that if the genus of a surface $\Sigma$ is at least $2|E(H)|$, then $\faces(H,\Sigma)\le 1$.
Furthermore, if at most one vertex of $H$ has degree at least four and $H$ can be embedded in $\Sigma$, then $\faces(H,\Sigma)\le 1$.
Therefore, the last case of the theorem applies only if $H$ has at least two vertices of degree at least four,
and the genus of all surfaces in the decomposition is less than $2|E(H)|$, improving on the bound $O(|V(H)|^4)$ of Grohe and Marx~\cite{gmarx}.

\medskip

Let us now discuss several consequences of Theorem~\ref{thm-main}.
\emph{Immersion} is a notion related to a topological minor: we say that $G$ \emph{immerses} $H$ if there
exists a mapping of vertices of $H$ to distinct vertices of $G$ (\emph{the branch vertices}) and of edges of $H$ to mutually edge-disjoint paths in $G$
joining the corresponding branch vertices.  The immersion is \emph{strong} if the internal vertices of the paths are distinct from the branch vertices.
Let $M_t$ be the graph consisting of vertices $w$, $w_i$, and $z_{ij}$ and edges $wz_{ij}$ and $w_iz_{ij}$ for $1\le i\le t$ and $1\le j\le t$.
Note that every simple graph with at most $t$ vertices is strongly immersed in $M_t$.
This implies that if $G$ contains $M_t$ as a topological minor, then every simple graph with at most $t$ vertices is strongly immersed in $G$. 
Furthermore, $M_t$ is planar and $\faces(M_t,\Sigma)=1$ for every surface $\Sigma$.
Consequently, we obtain the following structural result for graphs avoiding a fixed immersion.

\begin{corollary}\label{cor-imm}
For every graph $H$, there exist constants $D$ and $a$ with the following property.
Every graph $G$ that does not contain $H$ as a strong immersion can be expressed as a clique-sum of
graphs $G_1$, $G_2$, \ldots, $G_n$ such that for $1\le i\le n$, all but at most $a$ vertices of $G_i$ have degree at most $D$.
\end{corollary}

However, Corollary~\ref{cor-imm} is not the best possible characterization for graph avoiding an
immersion---in that context, joins over small edge cuts are the natural combination operation,
rather than clique-sums.  Such a characterization with a tree-like decomposition over edge cuts
(and also with significantly better constants) was obtained by Seymour and Wollan~\cite{seywol} and independently
by DeVos et al.~\cite{mattpriv}.

This motivates a question whether there exists a natural graph parameter corresponding to the decomposition by clique sums into pieces
of almost bounded maximum degree, as in Corollary~\ref{cor-imm}.  This parameter turns out to be \emph{$\infty$-admissibility}
(also referred to as \emph{infinite degeneracy} by Richerby and Thilikos~\cite{thilinfdeg}) defined as follows.

Given an ordering $v_1$, $v_2$, \ldots, $v_n$ of vertices of a graph $G$, the \emph{$d$-backconnectivity} of $v_k$
is the maximum number of paths of length at most $d$ from $v_k$ to $\{v_1, \ldots, v_{k-1}\}$ that intersect one another only in $v_k$.
The \emph{$d$-admissibility} of the ordering is the maximum of the $d$-backconnectivities of the vertices of $G$.
The $d$-admissibility of $G$ is the minimum $d$-admissibility over all orderings of $V(G)$.

The $1$-admissibility is also known as the \emph{degeneracy} of a graph, and the $2$-admissibility is a parameter
(generally referred to just as the \emph{admissibility}) related to the arrangeability of a graph, see e.g.~\cite{arr2,arr3,arr1}.
The $d$-admissibility is also related to generalized coloring numbers~\cite{kierstead2003orderings,zhu2009colouring}.
For us, the case that $d=\infty$ (i.e., we do not put any restrictions on the lengths of the paths) is relevant.

\begin{corollary}\label{cor-adm}
For every $t$, there exist constants $D$ and $a$ with the following property.
Every graph $G$ with $\infty$-admissibility at most $t$ can be expressed as a clique-sum of
graphs $G_1$, $G_2$, \ldots, $G_n$ such that for $1\le i\le n$, all but at most $a$ vertices of $G_i$ have degree at most $D$.
\end{corollary}

Furthermore, a weak converse holds, showing that the $\infty$-admissibility is indeed the right parameter
described by such a decomposition.
\begin{theorem}\label{thm-converse}
For every $D$ and $a$, there exists $T$ such that if $G$ is a clique-sum of
graphs with at most $a$ vertices of degree greater than $D$, then $G$ has $\infty$-admissibility at most $T$.
\end{theorem}

The proofs of Corollary~\ref{cor-adm} and Theorem~\ref{thm-converse} are given in Section~\ref{sec-appl}.

\section{Proof outline}

The proof of Theorem~\ref{thm-main} uses advanced concepts and ideas from the
graph minor structure theory of Robertson and Seymour, as well as the proof of Grohe and Marx~\cite{gmarx}, and it is not possible
to adequately explain this theory (derived in a series of more than 20 long papers) on just a few pages.  Consequently, the sections~\ref{sec-def}--\ref{sec-struct}
are aimed at a reader familiar with the theory, and a casual reader may find them impossible to understand.
In particular, we use standard definitions from the area without
an explanation, although we provide references to the papers from which they originate.
For this reason, let us provide an outline explaining the basic ideas of the proof in a more accessible way.

The proof of Grohe and Marx~\cite{gmarx} essentially works as follows.  Consider a graph $G$
that does not contain $H$ as a topological minor.  We gradually find a decomposition
of $G$ into a tree whose nodes correspond to the pieces of $G$ belonging to one of the cases
described in the statement of Theorem~\ref{thm-gmarx}.  We first choose a root piece of the
tree, split the rest of the graph to connected components attaching to this piece, and proceed
with each such component $G'$ recursively.

At this point $G'$ has a specified set of vertices $X$ of bounded size,
representing its intersection with the rest of the graph.  If there is a small separator $S$
in $G'$ splitting it to $G'_1$ and $G'_2$, such that the intersections of $G'_1$ and $G'_2$
with $X$ have roughly equal size, we introduce a new node corresponding to the subgraph
of $G'$ induced by $X\cup S$, and process $G'_1$ and $G'_2$ recursively.

Otherwise, $X$ is well-connected in $G'$.  
If $G'$ contains a minor of $K_n$ that is well-connected to $X$ in a similar sense (where $n\approx |V(H)|^2$ is a properly chosen constant), then since $G'$ does not contain a topological
minor of $H$, one can argue that almost all vertices of high degree in $G'$ must be split off
from $X$ by small separators (otherwise, we could link the high-degree vertices through
the minor of $K_n$ in order to obtain a topological minor of $H$).  We cut the graph on these small separators and replace them by
cliques, and argue that the piece containing $X$ contains only constantly many vertices of high degree.
We then process the cut-off parts recursively.

Finally, if no minor of $K_n$ is well-connected to $X$, one can apply a similar idea to the small
separators of $G'$ separating $X$ from the minors of $K_n$ in $G'$, and obtain a decomposition
with the central piece avoiding the minor of $K_{n'}$ for some $n'=O(n)$. This central piece can be further
split by Theorem~\ref{thm-robsey} to obtain a decomposition into pieces almost embedded
in surfaces of genus $O(n^2)$.

We proceed on from there, investigating when a topological minor of $H$ appears in an almost embedded graph $G'$.
We use a well-known fact that the pieces of the Robertson-Seymour decomposition
either have a bounded number of vertices, or are ``generic'' in the sense that the embedding has high
representativity and the vortices are mutually far apart.  In the latter case, Robertson and Seymour~\cite{rs7}
gave sufficient conditions for existence of a topological minor $H$ with prescribed branch vertices.
Essentially, if an embedding of $H$ is topologically possible and the branch vertices
are sufficiently connected and sufficiently far apart in $G'$, then $H$ appears as a topological minor in $G'$.

We use this claim to derive that if such a generic almost embedding in $\Sigma$ does not contain $H$ as a topological
minor, but $H$ can be embedded in $\Sigma$,
then all but constantly many vertices of high degree are either cut off by small separators,
or near to one of less than $\faces(H,\Sigma)$ faces of the embedding.  If $H$ is planar, we further
apply the same claim recursively to the parts split off by small separators, eventually concluding
that $G'$ has the structure described in the last case of Theorem~\ref{thm-main}.

Let us remark that if $\faces(H,\Sigma)=1$, then the conditions of the last case imply that the piece
is obtained by patching a graph $G''$ with bounded maximum degree.  In this case, we can process the
patches containing high-degree vertices recursively and attach them to $G''$ through clique-sums,
ending up with a central piece with the structure described in the first case of Theorem~\ref{thm-main},
instead.  Thus, the last case of Theorem~\ref{thm-main} applies only if $\faces(H,\Sigma)\ge 2$.

In the following section, we introduce the mentioned auxiliary results from graph minor theory more precisely.
The sufficient conditions for existence of a topological minor are derived Section~\ref{sec-top}.
Using them, the rest of the proof (presented in Section~\ref{sec-struct}) is straightforward, although somewhat technical.

\section{Notation and auxiliary results}\label{sec-def}

All graphs that we consider are simple (this does not affect the generality of our results, since
topological minor containment for multigraphs can be equivalently transformed to simple graphs by subdividing edges).
For definitions of a \emph{(respectful) tangle (controlling a minor)}, the metric derived from a respectful tangle, the function (\emph{slope}) $\ins$
defined by a respectful tangle, $p$-vortex,
\emph{($\TT$-central) segregation (of type $(p,k)$)}, an \emph{arrangement} of a segregation, see
Robertson and Seymour~\cite{rs12, robertson2003graph}.

We mostly use the notation of Grohe and Marx~\cite{gmarx}, in particular, we use $N^G[X]$ for the closed neighborhood
of a set $X$ in $G$, $N^G(X)=N^G[X]\setminus X$, and $S^G(X)$ for the separation $(A, B)$ of $G$ such that $A$ is the subgraph of $G$ induced by $N^G[X]$
and $V(A\cap B)=N^G(X)$.
The separation $(A,B)$ \emph{removes} $X$ if $X\subseteq V(A)\setminus V(B)$.  When a graph $G$ and a tangle $\TT$ are fixed
and there exists a separation in $\TT$ that removes $X$, we let $W(X)$ denote the set $W$ such that $X\subseteq W$,
$S^G(W)\in \TT$, the order of $S^G(W)$ is minimum and subject to that, $|W|$ is minimum (such a set exists and it is unique
by Lemma 5.7 of \cite{gmarx}).  A vertex $v$ is \emph{$n$-free} if $n$ is less than the order of $\TT$ and
either no separation in $\TT$ removes $\{v\}$ or the order of $S^G(W(\{v\}))$ is at least $n$; and in particular, $\deg(v)\ge n$.

A \emph{tree decomposition} of a graph $G$ is a pair $(T,\beta)$, where $T$ is a rooted tree and $\beta$
is a function assigning to each vertex of $T$ a subset of $V(G)$, such that for each $v\in V(G)$,
the set $\{t\in V(T):v\in \beta(t)\}$ is nonempty and induces a connected subtree of $T$,
and for each edge $uv\in E(G)$, there exists $t\in V(T)$ with $u,v\in\beta(t)$.  A \emph{star decomposition}
is a tree decomposition $(T,\beta)$ such that $T$ is a star.  The root $s$ of a star decomposition is
called its \emph{center}.  The sets $\beta(t)$, where $t\in V(T)$ is not the center, are called the \emph{tips}.
As in~\cite{gmarx}, we define the \emph{torso} $\tau(s)$ of the center of the decomposition to be the graph
obtained from the subgraph of $G$ induced by $\beta(s)$ by adding a clique with vertex set $\beta(s)\cap \beta(t)$ for each tip $t$.
For a tip $t$, let $(A_t,B_t)$ be the separation of $G$ such that $A_t$ is the subgraph of $G$ induced by $\beta(t)$
and $V(B_t)=V(G)\setminus (\beta(t)\setminus\beta(s))$.
For a tangle $\TT$, we say that the star decomposition is \emph{$\TT$-respecting} if every tip $t$ satisfies $(A_t,B_t)\in \TT$.

We use the structure theorem of Robertson and Seymour~\cite{robertson2003graph}.  However, we will need a formulation
that is slightly stronger than the one presented in~\cite{robertson2003graph}.  Let us first introduce a few definitions.

Let $\TT$ be a tangle in a graph $G$ and suppose that $S$ is a $\TT$-central segregation of $G$ of type $(p,k)$
with an arrangement in a surface $\Sigma$ without boundary.  For $s\in S$, let $\partial s$ denote the boundary vertices of $s$.
We call the elements of $S$ with boundary of size at most three \emph{cells} and the remaining (at most $k$) elements \emph{$p$-vortices}.
We say that $S$ is \emph{maximal} if there exists no $\TT$-central segregation $S'$ of $G$ of type $(p,k)$ with an arrangement in $\Sigma$
such that $S$ and $S'$ have the same $p$-vortices, each cell of $S'$ is a subgraph of a cell of $S$ and the inclusion is strict for at least one cell.
If $S$ is maximal, then let $T(S)$ be the graph such that $V(T(S))=\bigcup_{s\in S}\partial s$ and
$E(T(S))=\{uv:s\in S, |\partial s|\le 3, u,v\in \partial s, u\neq v\}$.
Observe that the maximality of $S$ implies that if $\{u,v,w\}=\partial s$ for $s\in S$, then $s$ contains a path joining $u$ with $v$ disjoint from $w$,
as well as two paths from $w$ to $u$ and to $v$ intersecting only in $w$.
Consequently, if $H$ is triangle-free and it is a topological minor of $T(S)$, then $H$ is a topological minor of $G$.
Furthermore, $T(S)$ has an embedding in $\Sigma$ and the boundary vertices of each $p$-vortex appear in order in the boundary of some face of $T(S)$; we call such a face a \emph{vortex face}.
Note that since $S$ is maximal, $T(S)$ is a minor of $G$.  Each tangle $\TT'$ in a minor of $G$
determines an \emph{induced} tangle $\TT''$ of the same order in $G$ (see \cite{rs10}, (6.1)).  If $\TT''\subseteq \TT$, we say that $\TT'$ is \emph{conformal}
with $\TT$.

\begin{theorem}\label{thm-gm}
For any graph $H$, there exists an integer $k$ such that for any non-decreasing positive function $\phi$, there exist integers $\rho, a, \theta\ge 0$ with the following property.
Let $\TT$ be a tangle of order at least $\theta$ in a graph $G$ controlling no $H$-minor
of $G$. Then there exists $A\subseteq V(G)$ with $|A|\le a$ and a $(\TT-A)$-central maximal segregation $S$ of $G-A$
of type $(p,k)$ for some $p\le \rho$, which has an arrangement in a surface in which $H$ cannot be embedded.
Furthermore, $T(S)$ contains a respectful tangle $\TT'$ of order at least $\phi(p)$ conformal with $\TT-A$,
and if $f_1$ and $f_2$ are vortex faces of $T(S)$ corresponding to distinct $p$-vortices and $d'$ is the metric
defined by $\TT'$, then $d'(f_1,f_2)\ge\phi(p)$.
\end{theorem}

Theorem~\ref{thm-gm} is implied by the results of \cite{rs17}.  Perhaps more straightforwardly,
it can also be proved by a modification of the argument of \cite{robertson2003graph}, as follows.

\begin{proof}[Outline of a proof of Theorem~\ref{thm-gm}.]
We proceed as in the proof of the analogous claim 3.1 in \cite{robertson2003graph}, except that we increase $\theta$
so that the resulting segregation $S$ has the required properties.
The segregation $S$ and its arrangement are obtained in \cite{robertson2003graph} in the proof of 6.1 by applying a variant of
(8.4) of \cite{rs15} to a $\Sigma$-span $L$ whose order can be made arbitrarily large by choosing $\theta$ sufficiently large.
The proof of (8.4) in \cite{rs15} is based on (7.7) of \cite{rs15}, in which by choosing the parameters appropriately
($\theta$ large enough relatively to $\theta'$, taking the function $\phi$ into account), the points where the bridges $J_1$, \ldots, $J_k$ (that eventually give rise to $p$-vortices)
attach to $L$ can be assumed to be mutually far apart (relatively to $\phi(p)$) in the metric defined by the tangle of $L$.
The main part ${\cal S}_0$ of the arrangement $S$ is obtained in (7.4) of \cite{rs15} and it is compatible
with $L$.  Parts of $L$ around the attachment vertices of $J_1$, \ldots, $J_k$ are cleared out when the
parts of $S$ at $p$-vortices are created; however, this decreases the order of the tangle as well as the distances
according to it by at most some function of $\theta'$.  Therefore, the parameters can be chosen so that
there exists a $\Sigma$-span $L'$ of order at least $\phi(p)$ compatible with the arrangement of $S$,
and if $d''$ is the metric defined by the tangle $\TT''$ of $L'$, then the distances between the faces
of $L'$ containing the $p$-vortices according to $d''$ are at least $\phi(p)$.
Note that we can assume that $S$ is maximal.

We can define the tangle $\TT'$ in $T(S)$ of the same order as $\TT''$ as follows.
Consider a separation $(A,B)$ of $T(S)$.  If $s$ is a cell of $S$, then $\partial s$
induces a clique in $T(S)$, and thus either $\partial s\subseteq V(A)$ or $\partial s\subseteq V(B)$.
Let $A'$ be the union of cells $s\in S$ such that $\partial s\subseteq V(A)$ and $B'$ the union of
cells $s\in S$ such that $\partial s\not\subseteq V(A)$.  Let $A''=A'\cap L'$ and $B''=B'\cap L'$.
Note that $(A'',B'')$ is a separation of $L'$ whose order is at most the order of $(A,B)$.
We put $(A,B)$ to $\TT'$ if and only if $(A'',B'')\in \TT''$.
It is easy to see that $\TT'$ is respectful tangle in $T(S)$ conformal with $\TT-A$ of order at least $\phi(p)$
and that the corresponding distances between vortex faces are at least $\phi(p)$.
\end{proof}

For $p$-vortices, we use the following characterization, which is essentially (8.1) of \cite{rs9}.

\begin{lemma}\label{lemma-pvortex}
If $F$ is a $p$-vortex with boundary $v_1$, \ldots, $v_m$, then $G$ has a path decomposition with bags $X_1$, \ldots, $X_m$ in order,
such that $v_i\in X_i$ for $1\le i\le m$ and $|X_i\cap X_j|\le p$ for $1\le i<j\le m$.
\end{lemma}

We call a path decomposition satisfying the conditions of Lemma~\ref{lemma-pvortex} a \emph{standard decomposition of a $p$-vortex}.
The following result, which is (3.2) of \cite{rs12}, describes sufficient conditions for the existence of
a rooted minor in a graph embedded in a surface.  A set $X$ is \emph{free} with respect to
a tangle $\TT$ if $|X|$ is at most the order of $\TT$ and there exists no $(A,B)\in\TT$ of order less than $|X|$ with $X\subseteq V(A)$.

\begin{theorem}\label{thm-gmembed}
For every surface $\Sigma$ without boundary and integers $k$ and $z$,
there exists an integer $\theta$ such that the following holds.
Let $G$ be embedded in $\Sigma$, let $\TT$ be a respectful tangle
in $\Sigma$ of order at least $\theta$ and let $d$ be the metric defined by $\TT$.
Let $f_1$, \ldots, $f_k$ be faces of $G$ such that $d(f_i,f_j)\ge\theta$ for $1\le i<j\le k$.
Let $Z$ be a set of at most $z$ vertices such that each $v\in Z$ is incident with one of $f_1$, \ldots, $f_k$.
Assume that $Z\cap f_i$ is free for $1\le i\le k$ and let $\Delta_1,\ldots, \Delta_k\subset \Sigma$ be pairwise disjoint
closed disks such that $G\cap \Delta_i=Z\cap f_i$.  If $M$ is a forest with $Z\subseteq V(M)$ embedded in $\Sigma$ so that
$M\cap \Delta_i=Z\cap f_i$ for $1\le i\le k$, then there exists a forest $M'\subseteq G$ such that
$M'\cap (\Delta_1\cup\ldots\cup \Delta_k)=Z$ and two vertices of $Z$ belong to the same component of $M'$ if
and only if they belong to the same component of $M$.
\end{theorem}

We will also need an auxiliary result concerning the metric derived from a respectful tangle (see (9.2) in \cite{rs14}).
Let $G$ be a graph with a $2$-cell embedding in a surface $\Sigma$ and let $\TT$ be a respectful tangle in $G$ with metric $d$.
If $a$ is a vertex or a face of $G$, then a \emph{$t$-zone around $a$} is an open disk $\Lambda\subset \Sigma$
bounded by a cycle $C\subseteq G$ such that $a\subseteq\Lambda$ and $d(a,a')\le t$ for all atoms $a'$ of $G$
contained in $\overline{\Lambda}$.

\begin{lemma}\label{lemma-zone}
Let $G$ be a graph with a $2$-cell embedding in $\Sigma$, let $\TT$ be a respectful tangle in $G$ of order $\theta$
and let $d$ be the metric derived from $\TT$.
Let $a$ be a vertex or a face of $G$, and let $2\le t\le \theta-3$.  Then there exists a $(t+2)$-zone $\Lambda$ around $a$
such that every atom $a'$ of $G$ with $d(a,a')<t$ satisfies $a'\subseteq \Lambda$.
\end{lemma}

\emph{Clearing a zone} (i.e., removing everything contained inside it from the graph) does not affect the order of the tangle
or the distances significantly, see (7.10) of \cite{rs12}:

\begin{lemma}\label{lemma-clearing}
Let $\Lambda$ be a $t$-zone around some vertex or face of a graph $G$ embedded in $\Sigma$ with a respectful tangle $\TT$ of
order $\theta\ge 4t+3$ and let $d$ be the metric derived from $\TT$.  Let $G'$ be the graph obtained from $G$ by clearing $\Lambda$.
Then, there exists a unique respectful tangle $\TT'$ in $G'$ of order $\theta-4t-2$ defining a metric $d'$ such that
whenever $a',b'$ are atoms of $G'$ and $a,b$ atoms of $G$ with $a\subseteq a'$ and $b\subseteq b'$, then
$d(a,b)-4t-2\le d(a',b')\le d(a,b)$.  Furthermore, $\TT'$ is conformal with $\TT$.
\end{lemma}

\section{Finding a topological minor}\label{sec-top}

In this section, we prove a key result (Lemma~\ref{lemma-main}) that enables us to restrict the structure of a graph $G$ without topological minor $H$,
where $G$ is almost embedded in a surface in that $H$ can be embedded.  First, let us argue that we can assume that $H$ is embedded
in a nice way.  An embedding of a graph $H'$ in a surface $\Sigma$ is \emph{nice} if it is closed $2$-cell
and contains a set $S$ of faces with $|S|=\faces(H',\Sigma)$ such that each vertex of $H'$ of degree at least four is incident with exactly one face of $S$.

\begin{lemma}\label{lemma-niceh}
Let $H$ be a graph that can be embedded in a surface $\Sigma$.  Then there exists triangle-free graph $H'$ containing $H$ as a topological minor
such that $\faces(H',\Sigma)=\faces(H,\Sigma)$ and $H'$ has a nice embedding in $\Sigma$.
\end{lemma}
\begin{proof}
Consider an embedding of $H$ in $\Sigma$ with a set $S$ of faces of $H$ dominating $(\ge\!4)$-vertices
such that $|S|=\faces(H,\Sigma)$.  If $H$ has a face $f$ not homeomorphic to an open disk,
then there exists a simple curve $c\subseteq f$ joining distinct points inside some edges $e_1$ and $e_2$ incident with $f$,
such that $c$ does not separate $f$.  We subdivide $e_1$ and $e_2$ and add an edge drawn along $c$.
We iterate this procedure until every face is an open disk (the number of iterations is bounded by a function of $H$ and $\Sigma$).
Next, we eliminate multiple incidences of vertices with faces by performing the same operation on pairs of edges incident with such a vertex,
thus obtaining a closed $2$-cell embedding.
Using the same transformation, we eliminate incidences of vertices of degree at least four with multiple faces in $S$.  Finally, we subdivide the
edges of the graph, obtaining a triangle-free graph $H'$ with a nice embedding in $\Sigma$.  By the construction, $H$ is a topological minor
of $H'$ and $\faces(H',\Sigma)\le\faces(H,\Sigma)$.  On the other hand, since $H$ is a topological minor of $H'$, we have $\faces(H',\Sigma)\ge\faces(H,\Sigma)$.
\end{proof}

Next, we prove a lemma on vortices.
If $X_1$, \ldots, $X_m$ are bags of a path decomposition of a graph $F$ in order, then for $v\in V(F)$, let $I(v)$ denote
the interval $\{i,i+1,\ldots, j\}$, where $i$ is the minimum index such that $v\in X_i$ and $j$ is the maximum index such that $v\in X_j$;
and let $B(v)=(X_i\cap X_{i-1})\cup (X_j\cap X_{j+1})$, where we set $X_0=X_{m+1}=\emptyset$.
If $F$ is a $p$-vortex with boundary vertices $v_1$, \ldots, $v_m$ (where $v_i\in X_i$ for $1\le i\le m$),
then we let $X(v_i)$ denote the set $\{v_i\}\cup (X_i\cap X_{i-1})\cup (X_i\cap X_{i+1})$.
When necessary to avoid ambiguity, we write $I_F(v)$, $B_F(v)$ and $X_F(v)$, respectively.

Suppose that $G=G_0\cup G_1\cup\ldots\cup G_k$, where $G_1$, \ldots, $G_k$ are $p$-vortices.  For $1\le i\le k$, we say that $G_i$ is \emph{circumscribed}
if there exists a cycle $v_1\ldots v_m\subseteq G_0$, where $v_1$, \ldots, $v_m$ are the boundary vertices of $G_i$ in order.
\begin{lemma}\label{lemma-vortfree}
Let $G=G_0\cup G_1\cup \ldots \cup G_k$, where $G_1$, \ldots, $G_k$ are vertex-disjoint circumscribed $p$-vortices.
For every integer $n\ge 0$, there exists $n'\ge 0$ such that the following holds.
Let $\TT_0$ be a tangle in $G_0$ of order $\phi>n'$ and let $\TT$ be the tangle of the same order induced in $G$ by $\TT_0$.
Let $v_1$, \ldots, $v_m$ be the vertices of $V(G_0)\cap V(G_1)$
and let $X_1$, \ldots, $X_m$ be bags of a standard path decomposition of $G_1$ in order, where $v_i\in X_i$ for $1\le i\le m$.
If $v\in V(G_1)\setminus V(G_0)$ is $n'$-free in $\TT$, then there exists a set $Y\subseteq V(G_0)\cap V(G_1)$ of size $n$
that is free with respect to $\TT_0$ and paths $P_1, \ldots, P_n\subseteq G_1$ disjoint from $B(v)$ joining $v$ to the vertices of $Y$ and
intersecting only in $v$.
\end{lemma}
\begin{proof}
Let $n'=2(2p+1)(n+2p)$.  
Let us recall that free sets with respect to a tangle of order $\phi$ form independent sets of a matroid of rank $\phi$ (see \cite{rs10}, section 12).
Therefore, there exists a set $Z\subseteq V(G_0)$ of size $\phi$ free with respect to $\TT_0$; clearly, $Z$ is also free with respect to $\TT$.
Note that there exists no separation $(A,B)$ of $G$ of order less than $n'$ such that $v\in V(A)\setminus V(B)$ and $Z\subseteq V(B)$,
since $(A,B)\not\in \TT$ as $v$ is $n'$-free and $(B,A)\not\in \TT$ as $Z$ is free and $n'<\phi$.  Therefore, there exist paths $Q_1, Q_2, \ldots, Q_{n'}\subset G$
joining $v$ to distinct vertices of $Z$ and intersecting only in $v$.  For $1\le i\le n'$, let $q_i\in V(G_0)$ be the vertex of $Q_i$ such
that the subpath of $Q_i$ between $v$ and $q_i$ contains no vertex of $G_0$ other than $q_i$.  Let $Y'=\{q_1,q_2,\ldots, q_{n'}\}$.

Consider a separation $(A,B)$ of $G_0$ such that $Y'\subseteq V(A)$ and $Z\subseteq V(B)$.
Let $S$ be the set obtained from $V(A)\cap V(B)$ by replacing each vertex $v\in V(A)\cap V(B)$ belonging to one of the $p$-vortices $G_1$, \ldots, $G_k$ by the set $X(v)$.
Observe that $S$ separates $Y'$ from $Z$ in $G$: otherwise, there exists a path $P$ from $Y'$ to $Z$ avoiding $S$,
and since $V(A)\cap V(B)\subseteq S$, $P$ contains a subpath inside $G_i$ for some $1\le i\le k$ joining a vertex $p_1\in V(A)\setminus V(B)$
with a vertex $p_2\in V(B)\setminus V(A)$.  However, both paths between $p_1$ and $p_2$ in the cycle in $G_0$ circumscribing $G_i$
must contain a vertex of $V(A)\cap V(B)$, and thus the subpath of $P$ between $p_1$ and $p_2$ in $G_i$ intersects $S$.
This is a contradiction, hence $S$ indeed separates $Y'$ from $Z$ in $G$.
Let $(A',B')$ be the corresponding separation of $G$ with $V(A')\cap V(B')=S$, $Y'\subseteq V(A)$ and $Z\subseteq V(B)$.  By the existence of the paths $Q_1$, \ldots, $Q_{n'}$, the order of $(A',B')$ is at least $n'$,
and thus the order of $(A,B)$ is at least $n'/(2p+1)=2(n+2p)$.

We conclude that there exist $2(n+2p)$ vertex-disjoint paths from $Y'$ to $Z$ in $G_0$.
Let $Y''$ be the set of endvertices of these paths in $Y'$ and $Z'$ the set of their endvertices in $Z$.  Since $Z$ is free with respect to $\TT_0$, $Z'$ is free as well.
Consider a separation $(C,D)\in \TT_0$ with $Y''\subseteq V(C)$, and let $c$ be the order of $(C,D)$.  Because of the paths between $Y''$ and $Z'$,
at least $|Z'|-c$ of the vertices of $Z'$ belong to $V(C)$, and since $Z'$ is free, we have $|Z'\cap V(C)|\le c$.  Therefore, $c\ge |Z'|/2=n+2p$.
This implies that the rank of $Y''$ in the matroid of free sets of $\TT_0$ is at least $n+2p$,
hence $Y''$ contains a subset $Y'''$ of size $n+2p$ that is free in $\TT_0$.

By the choice of $Y'$, there exist paths $P_1, \ldots, P_{n+2p}\subseteq G_1$
joining $v$ with the vertices of $Y'''$ and intersecting only in $v$.  At most $2p$ of these paths intersect $B(v)$, and thus we can choose a set
$Y\subseteq Y'''$ of size $n$ satisfying the conclusions of Lemma~\ref{lemma-vortfree}.
\end{proof}

The following construction enables us to assume that the $p$-vortices are circumscribed, and it will also be useful later in the paper.
Let $G=G_0\cup G_1$, where $G_0$ is embedded in a surface $\Sigma$ and $G_1$ is either a $p$-vortex such that there exists a closed
disk $\Delta\subseteq\Sigma$ intersecting the embedding of $G$ exactly in the boundary vertices of $G_1$ in order,
or a single vertex.  In the latter case, we consider $G_1$ to be a $0$-vortex and let $\Delta\subseteq \Sigma$ be a closed disk intersecting $G_0$ only in $G_1$,
such that the vertex $G_1$ is contained in $\bd(\Delta)$.  Let $\TT$ be a respectful tangle of order $\theta$ in $G_0$, and let $t\ge 2$ be an integer such that
$\theta\ge 8t+11$.  If there exists a simple closed $G_0$-normal curve $c$ intersecting $G_0$ in at most $t$ points such that $\Delta\subseteq \ins(c)$, then
choose $c$ so that $\ins(c)$ is maximal and let $\Delta_0=\ins(c)$; otherwise, let $\Delta_0=\Delta$.
Let $S$ be the set of vertices $w$ such that there exists a simple $G_0$-normal curve joining $w$ to $V(G_0)\cap \bd(\Delta_0)$ intersecting $G_0$
in less than $t$ points.  There exists a component $K$ of $G_0$ such that $(G_0-V(K), K)\in \TT$, and since $\TT$ is respectful,
the embedding of $K$ in $\Sigma$ is $2$-cell.  By Lemma~\ref{lemma-zone}, there exists a $(2t+2)$-zone $\Lambda\subset \Sigma$ such that $S\subset \Lambda$.
Therefore, the face $f$ of $G_0-S$ that contains $\Delta_0$ is a subset of $\Lambda$.  Observe that there exists a cycle $C$ in the boundary of $f$
such that the open disk $\Lambda'\subseteq\Lambda$ bounded by $C$ contains $f$.  Let $G'_0$ be the graph obtained from $G_0$ by clearing the $(2t+2)$-zone $\Lambda'$
and let $G'_1$ be the subgraph of $G$ such that $G=G'_0\cup G'_1$, $G'_0$ and $G'_1$ are edge-disjoint and $V(G'_0)\cap V(G'_1)=V(C)$.
Let $\partial G'_1$ consist of the vertices of $C$ in the order along $C$.  We call $G'_1$ a \emph{$t$-extension of $G_1$.}
Note that $G'_1$ is circumscribed, and that $\partial G'_1\cap V(G_1)=\emptyset$.

\begin{lemma}\label{lemma-vortext}
If $G'_1$ is a $t$-extension of a $p$-vortex, then $G'_1$ is a $(3t+4p)$-vortex.
\end{lemma}
\begin{proof}
Let $\Lambda'$ be the disk in that $G'_1$ is embedded with a $p$-vortex $G_1$, and consider any partition of $\partial G'_1$ to two arcs
$A$ and $B$. Let $a_1$ and $a_2$ be the endpoints of $A$, and for $i\in\{1,2\}$, let $c_i$ be a $G'_1$-normal simple curve drawn in $\Lambda'$
intersecting $G'_1$ in less than $t$ vertices and joining $a_i$ with a vertex $v_i\in \bd(\Delta_0)$.  If $\Delta=\Delta_0$ in the construction
of $G'_1$, then let $Z=((c_1\cup c_2)\cap G'_1)\cup X(v_1)\cup X(v_2)$, otherwise let $Z=((c_1\cup c_2)\cap G'_1)\cup (V(G'_1)\cap \bd(\Delta_0))$.
Observe that $Z$ separates $A$ from $B$ in $G'_1$ and that $|Z|<2t+\max(4p,t)$.  Since the choice of $A$ and $B$ was arbitrary, this shows that
$G'_1$ contains no transaction of order $3t+4p$ as required.
\end{proof}

Finally, we can prove the key lemma.

\begin{lemma}\label{lemma-main}
Let $H$ be a graph with a nice embedding in a surface $\Sigma$ without boundary, and let $p$ and $k$ be integers.  There exist constants $\phi$, $n'$ (where $\phi>n'$) and $a$
with the following property.  Let $G=G_0\cup G_1\cup \ldots \cup G_k$, where $G_0$ is embedded in $\Sigma$ and $G_1$, \ldots, $G_k$ are $p$-vortices
attaching to boundaries of distinct faces of the embedding of $G_0$.  Let $\TT_0$ be a respectful tangle in $G_0$ of order at least $\phi$
and let $\TT$ be the tangle of the same order induced in $G$.  Let $d_0$ be the metric derived from $\TT_0$ and assume that 
the distance according to $d_0$ between every two vortex faces is at least $\phi$.
If $H$ is not a topological minor of $G$, then there are at most $\faces(H,\Sigma)-1$ indices $i\in \{1,\ldots, k\}$ such that at least $a$ vertices of $V(G_i)$
are $n'$-free with respect to $\TT$; and in particular, $\Delta(H)\ge 4$.
\end{lemma}
\begin{proof}
Let $b=\faces(H,\Sigma)$ and let $F=\{f_1,\ldots, f_b\}$ be the set of faces of the nice embedding of $H$ dominating $(\ge\!4)$-vertices.
Let $a=(8p+12)|V(H)|$, let $n'$ be chosen so that Lemma~\ref{lemma-vortfree} applies with $n=\Delta(H)|V(H)|$ and $(4p+6)$-vortices
and so that $n'>8p+13$,
let $\theta$ be the constant given by Theorem~\ref{thm-gmembed} for $\Sigma$, $k=b+|E(H)|$ and $z=4|E(H)|$
and let $\phi=\max(n'+1,(\theta+7)(2+|E(H)|))+26k$.  Suppose on the contrary that $G$ satisfies the assumptions of the lemma, but
each of the $p$-vortices $G_1$, \ldots, $G_b$ contains a set $A_i$ of $n'$-free vertices with $|A_i|\ge a$.  Let $G'_1$, \ldots, $G'_k$ be
$2$-extensions of $G_1$, \ldots, $G_k$ and let $G'_0$ be the graph obtained from $G_0$ by clearing out the corresponding $6$-zones,
so that $G$ is the edge-disjoint union of $G'_0$, $G'_1$, \ldots, $G'_k$.
Let $\TT'_0$ be the respectful tangle in $G'_0$ of order $\phi'=\phi-26k$, such that the distance between any two vortex faces in the associated
metric $d'_0$ is $\phi'$, which exists by Lemma~\ref{lemma-clearing}.
Note that $\TT'_0$ is conformal with $\TT_0$, which in turn is conformal with $\TT$.  Since $\phi'>n'$, all vertices of $A_1\cup\ldots \cup A_b$ are $n'$-free in $G$ with respect
to the tangle of order $\phi'$ induced by $\TT_0'$.  For $1\le i\le b$, note that $G'_i$ is a circumscribed $(4p+6)$-vortex and $A_i$ is disjoint with $\partial G'_i$,
hence for each $v\in A_i$, there exists a set $F_v\subseteq \partial G'_i$ of size $n$ that is free with respect to $\TT'_0$, joined to $v$ by paths as described
in Lemma~\ref{lemma-vortfree}.

For some $i\in\{1,\ldots, b\}$, let us consider an auxiliary graph $L$ with the vertex set $A_i$ such that two vertices $u,v\in A_i$ are adjacent in $L$ iff
the intervals $I_{G'_i}(u)$ and $I_{G'_i}(v)$ intersect.  Since $n'>8p+13$ and $v$ is $n'$-free in $G$, we have $|I_{G'_i}(v)|\ge 2$ for $v\in A_i$.
Consider a clique $Q$ in $L$; there exists $j$ such that $j\in I_{G'_i}(v)$ for all $v\in V(Q)$,
i.e., all vertices of $Q$ belong to the $j$-th bag of the path decomposition of $G'_i$.
Since $|I_{G'_i}(v)|\ge 2$ for $v\in A_i$, each vertex of $Q$ also belongs either to the $(j-1)$-th or to the $(j+1)$-th bag of the path decomposition of $G'_i$.
Since $G'_i$ is a $(4p+6)$-vortex, it follows that $|V(Q)|\le 8p+12$.  Consequently, we have $\omega(L)\le 8p+12$, and since $L$
is an interval graph, it follows that $L$ has an independent set $M$ of size at least $\frac{|A_i|}{8p+12}\ge |V(H)|$.
Note that for distinct $m_1,m_2\in M$, the paths between $m_1$ and $F_{m_1}$ and between $m_2$ and $F_{m_2}$ are vertex-disjoint.

For each vertex $h\in V(H)$ incident with the face $f_i$, choose a distinct vertex $v_h$ in $M$.  Since the free sets are independent sets of a matroid
and each of the sets $F_v$ for $v\in A_i$ has rank at least $n$,
we can also choose a subset $F_h\subseteq F_{v_h}$ of size $\deg(h)$ for each $h\in V(H)$ incident with $f_i$ so that the union of all these sets
is free.  Furthermore, we can choose the vertices $v_h$ so that the order of the sets $F_h$ along $\partial G'_i$ matches the order of the
vertices $h$ along $f_i$.

Let $K$ be the component of $G'_0$ such that $(G'_0-V(K),K)\in \TT'_0$.  If $b>0$, let $e$ be an arbitrary edge of $K$ incident with the cycle circumscribing $G'_1$;
otherwise, let $e$ be an arbitrary edge of $K$.  By (8.12) of \cite{rs11}, there exists another
edge $f\in K$ such that $d'_0(e,f)=\phi'$.  Let $P$ be a path joining $e$ with $f$.  As in (4.3) of \cite{rs12}, we conclude that $P$ contains
a set $R$ of edges $e_1$, \ldots, $e_{|E(H)|}$ such that the distance between each pair of such edges is at least $\theta$, and their distance from $e$
is at least $\theta+2$, but at most $\phi'-\theta-2$ (consequently, their distance from any vortex face is at least $\theta$) and each two vertices
joined by one of these edges form a free set with respect to $\TT'_0$.  Let $R'$ be the set of vertices of $H$ incident with the faces in $F$.
By applying a homeomorphism, the embedding of $H$ can be altered to an embedding
$H'$ such that each edge of $H'$ contains an edge of $R$ as a subset and for every $h\in R'$ incident with $f_i$ for some $i\in \{1,\ldots, b\}$,
$h$ is drawn in the vortex face of $G'_i$ and each edge incident with $h$ contains a vertex of $F_h$.  Let $M$ be the subcubic forest obtained from the
embedding of $H'$ by removing from each edge the segment corresponding to $R$ and by removing the vertices of $R'$ together with the parts of edges
incident with them drawn inside the vortex faces; new vertices are added to $M$ at the ends of the removed segments, let their set be denoted by $N$.
By Theorem~\ref{thm-gmembed}, we conclude that $G'_0$ contains as a subgraph a forest $M'$ such that the vertices in $N$ are in the same
component of $M'$ if and only if they are in the same component of $M$.  We can assume that every leaf of $M'$ belongs to $N$, and thus $M'$ is subcubic.
Then,
$M'$ together with $R$ and together with the disjoint paths in the vortices joining the vertices $\{v_h:h\in R'\}$ with the corresponding vertices of $N$
form a topological minor of $H$ in $G$.  This is a contradiction.
\end{proof}

The following claim is proved in the same way:
\begin{lemma}\label{lemma-mainsim}
Let $H$ be a graph embedded in a surface $\Sigma$ without boundary, and let $p$ and $k$ be integers.  There exist constants $\phi$ and $n$ (with $n<\phi$)
with the following property.  Let $G=G_0\cup G_1\cup \ldots \cup G_k$, where $G_0$ is embedded in $\Sigma$ and $G_1$, \ldots, $G_k$ are $p$-vortices
attaching to boundaries of distinct faces of the embedding of $G_0$.  Let $\TT_0$ be a respectful tangle in $G_0$ of order at least $\phi$
and let $\TT$ be the tangle of the same order induced in $G$.  Let $d_0$ be the metric derived from $\TT_0$ and assume that 
the distance according $d_0$ between any two vortex faces is at least $\phi$.
If $H$ is not a topological minor of $G$, then there are at most $|V(H)|-1$ indices $i\in\{1,\ldots, k\}$ such that at least one vertex of $V(G_i)$
is $n$-free with respect to $\TT$.
\end{lemma}

From the proof of Lemma 6.12 of \cite{gmarx}, we can extract the following lemma regarding the existence of separators that
cut off all high-degree vertices.
\begin{lemma}\label{lemma-cutbig}
There exists a function $f(n)\in O(n^3)$ with the following property.
Let $G$ be a graph with a tangle $\TT$ of order $\theta$ and let $n$ be an integer with $f(n)<\theta$.
Let $X\subseteq V(G)$ be a set such no vertex in $V(G)\setminus X$ is $n$-free.
There exists a collection $\LL$ of vertex-disjoint subsets of $V(G)$ such that
\begin{itemize}
\item if $L_1,L_2\in \LL$ are distinct, then $N^G(L_1)\cap L_2=\emptyset$,
\item for every $L\in \LL$, the subgraph of $G$ induced by $L$ is connected and satisfies $|N^G(L)|\le f(n)$ and $S^G(L)\in \TT$, and
\item if $G'$ is the graph obtained from $G-\bigcup_{L\in \LL} L$ by adding cliques with vertex set $N^G(L)$ for each $L\in \LL$ (without creating parallel edges),
then every $v\in V(G')\setminus X$ satisfies $\deg_{G'}(v)\le nf(n)$.
\end{itemize}
\end{lemma}

The following technical lemma is standard (we include its proof for completeness).
\begin{lemma}\label{lemma-mkdist}
For every $t,n>0$ and every non-decreasing positive function $f$, there exists $T$ such that the following holds.  Let $Z$ and $U$ be sets of points of a metric space with metric $d$,
such that $|Z|\le n$ and for every $u\in U$ there exists $z\in Z$ with $d(u,z)<t$.  Then, there exists a subset $Z'\subseteq Z$ and a number $t'\le T$ such that
\begin{itemize}
\item for every $u\in U$, there exists $z\in Z'$ with $d(u,z)<t'$, and
\item for distinct $z_1,z_2\in Z'$, $d(z_1,z_2)\ge f(t')$.
\end{itemize}
Furthermore, if the elements of a set $Z''\subseteq Z$ are at distance at least $T$ from each other, then we can choose $Z'$ so that $Z''\subseteq Z'$.
\end{lemma}
\begin{proof}
Let $t_0=t$ and for $1\le i\le n-1$, let $t_i=t_{i-1}+f(t_{i-1})$; and set $T=t_{n-1}$.  We construct a sequence of sets $Z=Z_0\supset Z_1\supset \ldots \supset Z_{n'}$ with $n'<n$
such that for every $u\in U$ and $i\le n'$, there exists $z\in Z_i$ with $d(u,z)<t_i$, as follows: suppose that we already found $Z_i$.  If $d(z_1,z_2)\ge f(t_i)$
for every distinct $z_1,z_2\in Z_i$, then set $n'=i$ and stop.  Otherwise, there exist distinct $z_1,z_2\in Z_i$ such that $d(z_1,z_2)<f(t_i)$ and $z_2\not\in Z''$;
in this case, set $Z_{i+1}=Z_i\setminus \{z_2\}$.  Clearly, the set $Z'=Z_{n'}$ has the required properties.
\end{proof}

Combining these results, we can obtain the following description of plane graphs avoiding a given topological minor.
\begin{lemma}\label{lemma-patches}
For every graph $H$ and every $n_0\ge 0$, there exist constants $n$, $D$, $m$ and $a$ such that if $G$ is a graph embedded in a disk $\Delta$
with $|G\cap \bd(\Delta)|\le n_0$ and $H$ is not a topological minor of $G$, then $G$ is a subgraph of
an $(n,\faces(H,\Sigma(0,0,0)),D,m, a)$-patch.
\end{lemma}
\begin{proof}
The claim is obvious when $\faces(H,\Sigma(0,0,0))$ is infinite (i.e., $H$ is not planar), with $D=m=a=0$ and $n=n_0$.
Hence, assume that $b=\faces(H,\Sigma(0,0,0))$ is finite.
Furthermore, by Lemma~\ref{lemma-niceh}, we can assume that $H$ has a nice embedding in the sphere with a set of $b$ faces
dominating $(\ge\!4)$-vertices.   Let $\phi_1$ and $n_1$ be the constants given by Lemma~\ref{lemma-mainsim}
applied for $H$ with $p=0$ and $k=|V(H)|$.  Let $\phi_2$, $n_2$ and $a_2$ be non-decreasing functions of one variable
such that $\phi_2(p)$, $n_2(p)$ and $a_2(p)$ are greater or equal to the constants given by Lemma~\ref{lemma-main} applied for
graphs embedded in the sphere with at most $|V(H)|$ $p$-vortices and for the graph $H$.
Let $T$ be the constant of Lemma~\ref{lemma-mkdist} applied with $t=\phi_1$, $n=|V(H)|$ and the function $f_1(r)=\phi_2(3r)+|V(H)|(8r+10)$.
Let $f_2$ be the function of Lemma~\ref{lemma-cutbig}.
Let $a'=a_2(3T)$ and $n'=\max(n_1,n_2(3T),T+1)$.
Let $\theta=\max(\phi_1, f_1(T), n_0, f_2(n'))$, $n=3\theta-2$, $D=n'f_2(n')+n$, $m=T+n$ and $a=\max(4\theta-2,a'|V(H)|+n)$.

We prove a slightly stronger claim: if $\Delta'\subset \Delta$ is a disk with a $G$-normal boundary such that $|G\cap \bd(\Delta')|\le n$
and $G'$ is the subgraph of $G$ embedded in $\Delta'$, then $G'$ is a subgraph of an $(n,b,D,m,a)$-patch.
We proceed by induction and assume that the claim holds whenever $G'\neq G$.
If $G$ has less than $3\theta-2$ vertices, then $G$ is $(n,b,D,m,a)$-basic, since $a> 3\theta-2$.
Therefore, assume that $G$ has at least $3\theta-2$ vertices.
Let $\Sigma$ be the sphere
obtained by identifying the boundary of $\Delta$ with the boundary of a disjoint disk $\Delta_1$.
Let $\Delta_2\subset \Sigma$ be a closed disk such that $\Delta_1\subseteq \Delta_2$,
$\bd(\Delta_2)$ is $G$-normal, $\bd(\Delta)\cap G\subseteq \bd(\Delta_2)\cap G$, the interior of $\Delta_2$ contains no vertices of $G$ and $|\bd(\Delta_2)\cap G|=3\theta-2$.

Suppose now that there exists a simple closed $G$-normal curve $c\subseteq \Sigma$ with $|G\cap c|<\theta$
such that each of the two closed disks in $\Sigma$ bounded by $c$ contains at least $\theta$ vertices of 
$G\cap \bd(\Delta_2)$.  Consequently, each of the disks contains at most $|G\cap \bd(\Delta_2)|-\theta+|G\cap c|<3\theta-2$
vertices of $G\cap (\bd(\Delta_2)\cup c)$.  Hence, $\bd(\Delta_2)\cup (c\cap \Delta)$ cuts $\Delta$ into disks such that
the boundary of each of these disks intersects $G$ in less than $3\theta-2\le n$ vertices. By induction, the subgraphs of $G$ embedded in these disks are subgraphs
of $(n,b,D,m,a)$-patches.  Since $|G\cap (c\cup \bd(\Delta_2))|<4\theta-2\le a$, a supergraph of $G$ is obtained by pasting these patches
to a graph with at most $a$ vertices, which is $(n,b,D,m,a)$-basic.

Therefore, we can assume that no such curve $c$ exists.  Then, $G$ contains a respectful tangle $\TT$ of order $\theta$
defined by its slope $\ins$ as follows: if $q$ is a closed simple $G$-normal curve intersecting $G$ in less than $\theta$
points, then let $\ins(q)$ be the closed disk bounded by $q$ such that $\ins(q)$ contains less than $\theta$ vertices
of $G\cap \bd(\Delta_2)$.  There is only one such disk, since $|G\cap \bd(\Delta_2)|>2\theta-2$.
We need to verify that $\ins$ satisfies the slope axioms of~\cite{rs11}.  The axiom (i) is trivial and the axiom (ii) holds
since $|G\cap \bd(\Delta_2)|>3\theta-3$. Hence, $\ins$ defines a respectful tangle by (6.1) and (6.5) of \cite{rs11}
(the slope $\ins$ is clearly even).

Let $U$ be the set of vertices of $G$ that are
$n'$-free in $\TT$.  Let $Z_0$ be the maximal subset of $U$ such that $d(z_1,z_2)\ge \phi_1$ for
all distinct $z_1,z_2\in Z_0$, where $d$ is the metric derived from $\TT$.  Since $H$ is not a topological subgraph of $G$,
Lemma~\ref{lemma-mainsim} implies that $|Z_0|<|V(H)|$.  By Lemma~\ref{lemma-mkdist}, there exist $Z\subseteq Z_0$ and $r\le T$
such that $d(z_1,z_2)\ge f_1(r)$ for distinct $z_1,z_2\in Z$ and for each $u\in U$, there exists $z\in Z$ with $d(u,z)<r$.
Note that since each $u\in U$ is $n'$-free and $n'>r$, every closed walk $W$ of length less than $2r$ in the
radial drawing of $G$ such that $u\in \ins(W)$ satisfies $u\in V(W)$.  Therefore, for each $u\in U$, there exists $z\in Z$
and a simple $G$-normal curve $c$ with ends $u$ and $z$ intersecting $G$ in less than $r$ vertices.
By Lemma~\ref{lemma-vortext}, we can express $G$ as $G_0\cup G_1\cup \ldots \cup G_k$, where $k=|Z|$,
\begin{itemize}
\item $G_0$ is embedded in the sphere with a respectful tangle $\TT_0$ of order at least $\phi_2(3r)$ conformal with $\TT$,
\item for $1\le i\le k$, $G_i$ is a $3r$-vortex, and
\item the distance between their vortex faces is at least $\phi_2(3r)$, according to the metric $d_0$ derived from $\TT_0$.
\end{itemize}
Since $G$ does not contain $H$ as a topological minor, Lemma~\ref{lemma-main} implies that each of $G_b$, $G_{b+1}$, \ldots, $G_k$ contains at most $a'$ vertices of $U$.
Let $A_1=U\cap (V(G_b)\cup V(G_{b+1})\cup \ldots\cup V(G_k))$ and let $S=Z\cap (V(G_1)\cup \ldots\cup V(G_{b-1}))$.
We have $|S|=b-1$.
Note that for each $u\in U\setminus A_1$, there exists $s\in S$ and a $G$-normal curve $c$ joining $u$ with $s$ intersecting $G$ in less than $T$ vertices.
Furthermore, we can alter $c$ so that it lies within $\Delta$, while adding at most $n$ new intersections with $G$.
Therefore, for each $u\in U\setminus A_1$, there exists $s\in S$ and a $G$-normal curve $c\subseteq \Delta$ joining $u$ with $s$
and intersecting $G$ in less than $T+n=m$ vertices.
Let $A_2=G\cap \bd(\Delta)$.

Let $\LL$ be the set from Lemma~\ref{lemma-cutbig}, applied with $n=n'$ and $X=U$.  Note that we can assume that for any $L\in \LL$,
if $c$ is a simple closed $G$-normal curve with $c\cap G\subseteq N^G(L)$, then $\ins(c)\cap V(G)\subseteq N^G[L]$.  Consequently,
for each $L\in \LL$, there exists a closed disk $\Delta_L\subset \Sigma$ with $G$-normal boundary such that $\Delta_L=\ins(\bd(\Delta_L))$,
$V(G)\cap \Delta_L=N^G[L]$, and $\bd(\Delta_L)\cap G=N^G(L)$.  In particular, $|\bd(\Delta_L)\cap G|\le f_2(n')$ for $L\in\LL$.

Let $\DD$ be the set of disks obtained from $\{\Delta_L:L\in\LL\}$ by intersecting with $\Delta$ (some of the disks $\Delta_L$ can be split
into several disks belonging to $\DD$).
Since $\Delta_L=\ins(\bd(\Delta_L))$ for every $L\in \LL$, each disk $\Delta_L$ contains less then $\theta$ vertices of $\bd(\Delta_2)\cap G\supseteq \bd(\Delta)\cap G$,
and thus $|\bd(\Pi)\cap G|< f_2(n')+\theta < n$ for every $\Pi\in \DD$.
By induction, the graph $G\cap \Pi$ is an $(n,b,D,m,a)$-patch.
Furthermore, consider the graph $G'$ obtained from $G$ by, for each $\Pi\in \DD$, removing vertices and edges contained in the interior of $\Pi$
and adding $\bd(\Pi)$ as a cycle.  Each of the added edges is either incident with a vertex of $A_2$ or belongs to a clique with
vertex set $N^G(L)$ for some $L\in \LL$.  Therefore, the degree of each vertex of $G'$ that does not belong to $U\cup A_2$ is at most
$n'f_2(n')+n=D$.

Let $S'$ be the set obtained from $S$ by replacing each vertex $s\in S\setminus V(G')$ by a vertex incident with the face of $G'$ inside that $s$
is drawn according to the embedding of $G$.
The sets $(A_1\cup A_2)\cap V(G')$ and $S'$ witness that $G'$ is $(n,b,D,m,a)$-basic.  Note that a supergraph of $G$ is obtained from $G'$ by
pasting $(n,b,D,m,a)$-patches.  Therefore, $G$ is a subgraph of an $(n,b,D,m,a)$-patch.
\end{proof}

\section{The structure}\label{sec-struct}

We start with a key ingredient for a local form of the structure theorem.

\begin{lemma}\label{lemma-sublocal}
For every graph $H$, integers $k_0$ and $p_0$, and a surface $\Sigma$ in that $H$ can be embedded, there exist integers $a$, $k$, $p$, $D$, $m$, $n$ and $\phi$ with the following property.
Let $G$ be a graph with a tangle $\TT$ of order at least $\phi$ and let $S$ be a $\TT$-central maximal segregation of $G$ of type $(p_0,k_0)$
with an arrangement in $\Sigma$.  Furthermore, assume that $T(S)$ contains a respectful tangle $\TT_0$ of order at least $\phi$ conformal with $\TT$,
and if $f_1$ and $f_2$ are vortex faces of $T(S)$ corresponding to distinct $p$-vortices and $d_0$ is the metric
defined by $\TT_0$, then $d_0(f_1,f_2)\ge\phi$.

If $G$ does not contain $H$ as a topological minor, then there exists a $\TT$-respecting star decomposition $(T,\beta)$ of $G$ with center $s$ and a set $A\subseteq \beta(s)$ of size at most $a$
such that either the maximum degree of $\tau(s)-A$ is at most $D$, or
$\faces(H,\hat{\Sigma})\ge 2$ and there exists an outgrowth $Q$ of a graph embedded in $\Sigma$
by at most $k$ vortices of depth at most $p$ satisfying
\begin{itemize}
\item all vertices $v\in V(Q)$ with $\deg_{\tau(s)-A}(v)>D$ belong to one of the vortices,
\item less than $\faces(H,\hat{\Sigma})$ of the vortices contain such a vertex, and
\item $\tau(s)-A$ is an $(n,\faces(H,\Sigma(0,0,0)),D,m,a)$-expansion of $Q$.
\end{itemize}
\end{lemma}
\begin{proof}
Let $H_{\Sigma}$ be the nice embedding in $\Sigma$ derived from $H$ using Lemma~\ref{lemma-niceh}.
Let $k_1=|V(H_\Sigma)|$. Let $k=k_0+k_1$.  Let $\phi_2$ and $n_2$ be the constants
of Lemma~\ref{lemma-mainsim} applied with $0$-vortices, $k_1$, $\Sigma$ and $H_\Sigma$.
Let $\phi_3$, $a_3$ and $n_3$ be non-decreasing functions of one variable such that
$\phi_3(p)$, $a_3(p)$ and $n_3(p)$ are greater or equal to the constants of Lemma~\ref{lemma-main} applied to graphs with
at most $k$ $p$-vortices, for surface $\Sigma$ and the forbidden topological minor $H_\Sigma$.
Let $f_5$ be the function of Lemma~\ref{lemma-cutbig} and for an integer $r\ge 0$, let $n_r=\max(n_2,n_3(3r+4p_0),r)$.
Let $T_3$ be the constant of Lemma~\ref{lemma-mkdist} applied with $t=\phi_2$, $n=k$ and the function $f_3(r)=\phi_3(3r+4p_0)+(8r+4f_5(n_r)+20)k$.
Let $\phi_4=f_3(T_3)$, $a_4=a_3(3T_3+4p_0)$, $n_4=n_3(3T_3+4p_0)$ and $n'=\max(n_2,n_4,T_3)$.
Let $p_5=3f_5(n')+12T_3+16p_0$.
Let $n_5$, $D_5$, $m$ and $a_5$ be the constants of Lemma~\ref{lemma-patches} applied to the graph $H'$ obtained from $H$ by subdividing each
edge once, with $n_0=f_5(n')$.
Let $n=\max(f_5(n'), n_5)$, $a=\max(ka_4,a_5)$, $n''=n'+a$, $p=2p_5f_5(n')+1$ and $D=\max(pn''f_5(n''), D_5)$.
Let $\phi=\max(\phi_2,\phi_4+p)$.

Let $G_1$, \ldots, $G_{k'_0}$ be the $p_0$-vortices of $S$, where $k'_0\le k_0$.
Let $G_0=T(S)$, and let $G'=G_0\cup G_1\cup\ldots\cup G_{k'_0}$.  
Let $\TT'$ be the tangle of order $\phi$ induced by $\TT_0$ in $G'$.
Since $H_{\Sigma}$ is triangle-free and $H$ is a topological minor of $H_{\Sigma}$, we conclude that $G'$ does not
contain $H_{\Sigma}$ as a topological minor.  Let $U_2$ be the set of all vertices of $G'$ that are $n_2$-free in $\TT'$.
Let $Z$ be a maximal subset of $U_2\cap V(G_0)$ such that the distances among the vertices of $Z$ according to $d_0$
are at least $\phi_2$.  By Lemma~\ref{lemma-mainsim} applied to $G_0$, we have $|Z|<k_1$.  Let $Z_1$ be the set consisting of $Z$ and of the set $F$ of vortex faces of $G_0$.
By Lemma~\ref{lemma-mkdist}, there exists $Z_2\subseteq Z_1$ and $r\le T_3$ such that each element of $U_2\cap V(G_0)$ is at distance less than $r$ from $Z_2$,
while the distances among the elements of $Z_2$ are at least $f_3(r)$, and $F\subseteq Z_2$.  Let $k'=|Z_2|$ and $p_r=3r+4p_0$. Let us recall that $n_r=\max(n_2,n_3(p_r),r)$.
Let $G'=G_0'\cup G_1'\cup \ldots\cup G'_{k'}$, where $\{G'_i:1\le i\le k'\}$ are $p_r$-vortices obtained as $r$-extensions of the elements of
$Z_2$ and $G'_0$ is obtained from $G_0$ by clearing the corresponding $(2r+2)$-zones.
Let $\TT'_0$ be the tangle of order at least 
$\phi'=\phi-(8r+10)k\ge \phi_4+p-(8r+10)k\ge f_3(r)+p-(8r+10)k=\phi_3(p_r)+(4f_5(n')+10)k+p$ in $G'_0$ (conformal with $\TT_0$) and $d'_0$ the associated metric as given by Lemma~\ref{lemma-clearing}.
Note that the distance between any pair of vortex faces in $G'_0$ according to $d'_0$ is at least $f_3(r)-(8r+10)k=\phi_3(p_r)+(4f_5(n_r)+10)k$.

Let $U_4\subseteq U_2$ be the set of vertices of $G'$ that are $n_r$-free in the tangle induced by $\TT'_0$ in $G'$ (or, by conformality, equivalently in $\TT'$).
Consider a vertex $u\in U_4\cap V(G_0)$ and let $z\in Z_2$ be the element such that $d_0(u,z)<r$, i.e., there exists a closed walk $W$ in the radial drawing of $G_0$
of length less than $2r$ such that $u,v\in \ins(W)$.  Since $n_r\ge r$ and $u$ is $n_r$-free, we have $u\in V(W)$.  We conclude that $W$ is either a path joining
$u$ with $z$, or a path to a cycle surrounding $z$.  In both cases, $u$ belongs to the $r$-extension of $z$.
Consequently, $U_4\subseteq V(G'_1)\cup \ldots\cup V(G'_{k'})$.

Let $b=\faces(H,\Sigma)$.  Since $H_{\Sigma}$ is not a topological minor of $G'$, by Lemma~\ref{lemma-main} we can assume that $\Delta(H)\ge 4$ (and thus $b>0$) and
that $|V(U_4)\cap V(G'_i)|\le a_4$ for $b\le i\le k'$.  Let $A=\bigcup_{b\le i\le k'} V(U_4)\cap V(G'_i)$ and note that $|A|\le ka_4$.

If $b=1$, then every $n'$-free vertex of $G'$ belongs to $A$.  Note that no vertex of $V(G)\setminus V(G')$ is $4$-free with respect to $\TT$,
since the segregation $S$ is $\TT$-central.  Suppose that $G$ contains a vertex $v\not\in A$ that is $n'$-free with respect to $\TT$.
Since $n'\ge 4$, we have $v\in V(G')$, and since $v\not\in A$, $v$ is not $n'$-free with respect to $\TT'$.  Hence, there exists
a separation $(I',J')\in \TT'$ of $G'$ of order less than $n'$ with $v\in V(I')\setminus V(J')$.
Let $I$ be the subgraph of $G$ consisting of $I'-E(G_0)$ and of all cells $s\in S$ with $\partial s\subseteq V(I')$,
and let $J$ be the subgraph consisting of $J'-E(G_0)$ and of all cells $s\in S$ with $\partial s\not\subseteq V(I')$.
The boundary of each cell $s\in S$ induces a clique in $G_0$, and thus $\partial s\subseteq V(I')$ or $\partial s\subseteq V(J')$.
Therefore, $(I,J)$ is a separation of $G$ of the same order as $(I',J')$, and since $\TT'$ is conformal with $\TT$,
we have $(I,J)\in \TT$.  This contradicts the assumption that $v$ is $n'$-free with respect to $\TT$.
We conclude that all $n'$-free vertices of $G$ belong to $A$.
Consequently, all $n'$-free vertices of $G$ belong to $A$, and $|A|\le ka_4\le a$.
As in Lemma~6.12 of \cite{gmarx}, we conclude
that $G$ has a $\TT$-respecting star decomposition with center $s$ such that the maximum degree of $\tau(s)-A$ is at most $D$.

Therefore, assume that $b\ge 2$.  Let $\LL$ be the set obtained from Lemma~\ref{lemma-cutbig} applied to $G'$ with $n=n_r$ and $X=U_4$.
Since the order of the respectful tangle $\TT'_0$ is greater than $f_5(n_r)$,
we can assume that for every $L\in \LL$ and for every simple closed $G'_0$-normal curve $c$ that intersects
$G'_0$ in a subset of $N^{G'}(L)$ and no boundary vertex of $G'_1$, \ldots, $G'_{k'}$ appears in the interior of $\ins(c)$,
we have $\ins(c)\cap V(G'_0)\subseteq N^{G'}[L]$.
Similarly, if $c$ is a simple closed $G'_0$-normal curve $c$ intersecting $G'_0$ in a subset of $N^{G'}(L)$
such that $\ins(c)$ contains the vortex face of $G'_i$ for some $i\in\{1,\ldots, k'\}$, then
we can assume that $(\ins(c)\cap V(G'_0))\cup V(G'_i)\subseteq N^{G'}[L]$. Consequently, each $L\in \LL$ satisfies one of the following:
\begin{itemize}
\item[(i)] there exists a disk $\Delta^L\subset \Sigma$ with $G'_0$-normal boundary such that $\Delta^L=\ins(\bd(\Delta^L))$, 
$G'_0\cap \bd(\Delta^L)=N^{G'}(L)$, and $N^{G'}[L]=V(G'_0)\cap \Delta^L$,
\item[(ii)] there exists exactly one $i\in\{1,\ldots, k'\}$ such that $V(G'_i)\subseteq L$, there exists a disk $\Delta^L\subset \Sigma$ such that
$\bd(\Delta^L)$ is $G'_0$-normal, $\Delta^L=\ins(\bd(\Delta^L))$,
$\Delta^L$ contains the vortex face of $G'_i$, $G'_0\cap \bd(\Delta^L)=N^{G'}(L)$ and $V(G'_i\cup (G'_0\cap \Delta^L))=N^{G'}[L]$, or
\item[(iii)] there exists exactly one $i\in\{1,\ldots, k'\}$ such that $V(G'_i)\cap L\neq\emptyset$ and there exist disks $\Delta^L_1, \ldots, \Delta^L_q\subset \Sigma$
 with $G'_0$-normal boundaries, each containing at least one vertex of $\partial G'_i$, such that
$\Delta^L_j=\ins(\bd(\Delta^L_j))$ and $G'_0\cap \bd(\Delta^L_j)\subseteq N^{G'}(L)$ for $1\le j\le q$
and $\bigcup_{j=1}^q(V(G'_0)\cap \Delta^L_j)=N^{G'_0}[L]$.
It is possible that $q=0$, in which case $N^{G'}[L]\subseteq V(G'_i)\setminus \partial G'_i$.
\end{itemize}
Furthermore, the interiors of all the disks are pairwise disjoint.

For $1\le i\le k'$, let us define an open disk $\Lambda_i\subset \Sigma$ as follows.  If there exists $L\in\LL$ of type (ii) such that
$\Delta^L$ contains the vortex face of $G'_i$, then let $\Lambda_i$ be the interior of $\Delta^L$.  Otherwise, let $\Pi_i$ be
the union of the vortex face $h_i$ of $G'_i$ with interiors of all disks $\Delta^L_j$ for all $L\in \LL$ of type (iii) that intersect
$G'_i$.  Note that all atoms in $\Pi_i$ are at distance at most $f_5(n_r)$ from $h_i$ according to the
metric $d'_0$, and thus by Lemma~\ref{lemma-zone}, there exists an $(f_5(n_r)+2)$-zone $\Lambda'_i\supseteq\Pi_i$ around $h_i$.
In this case, we let $\Lambda_i\subseteq \Lambda'_i$ be an open disk such that $\Pi_i\subseteq\Lambda_i$, the boundary of $\Lambda_i$
is $G'_0$-normal and the part of $G'_0$ contained in the closure of $\Lambda_i$ is
as small as possible.  Note that we can choose $\Lambda_i$ so that for every $L\in\LL$ of type (i), the interior of $\Delta_L$
is either a subset of $\Lambda_i$ or disjoint with $\Lambda_i$.  See Figure~\ref{fig-extvort} for an illustration.

\begin{figure}
\includegraphics{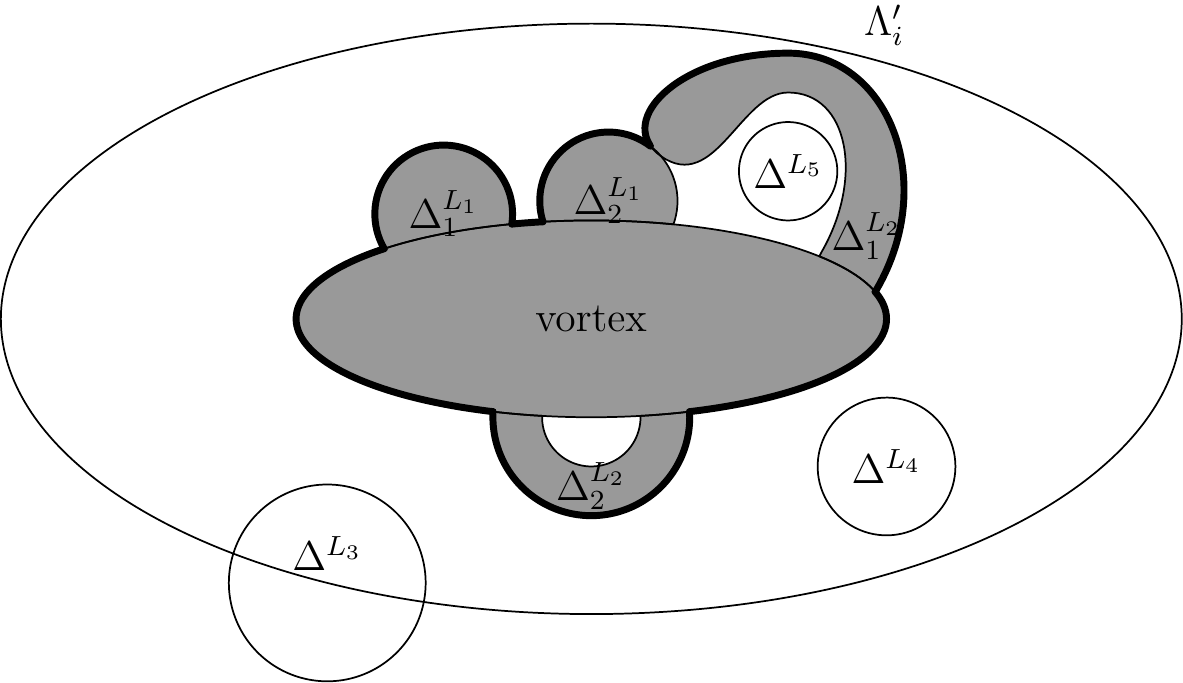}
\caption{Adding disks to a vortex. The gray area is the set $\Pi_i$. The boundary of $\Lambda_i$ is drawn by a thick line.}\label{fig-extvort}
\end{figure}

For $1\le i\le k'$, let $G''_i$ consist of $G'_i$ and of the subgraph of $G'_0$ contained in the closure of $\Lambda_i$,
with $\partial G''_i=\bd(\Lambda_i)\cap G'_0$.
Let $G''_0$ be the subgraph of $G'_0$ obtained by removing all vertices contained in $\bigcup_{i=1}^{k'}\Lambda_i$.
Note that $\Lambda_i$ is contained inside an $(f_5(n_r)+2)$-zone around the vortex face of $G'_i$ for $1\le i\le k'$.
By Lemma~\ref{lemma-clearing}, the subgraph of $G'_0$ obtained by clearing these zones contains a tangle $\TT_4$ of order $\phi''\ge\phi'-(4f_5(n_r)+10)k>p$
and the distances among the zones are at least $\gamma_r-(4f_5(n_r)+10)k>0$.  Consequently, the graphs $G''_1$, \ldots, $G''_{k'}$ are vertex-disjoint
and analogously to Lemma~\ref{lemma-vortext}, they are (not necessarily circumscribed) $p_5$-vortices.  Let
$\TT''_0$ be the tangle order $\phi''$ in $G''_0$ induced by $\TT_4$.
Note that $G'=G''_0\cup G''_1\cup\ldots\cup G''_{k'}$ and that for each $L\in \LL$, there exists $i\in\{0,1,\ldots, k'\}$ such that $N^{G'}[L]\subseteq V(G''_i)$.

Let us recall that $H'$ is a subdivision of $H$.
The graph $H'$ is not a topological minor of $G''_0$, since $H$ is a topological minor of $H'$ and $H'$ is triangle-free.
Consequently, if $L\in \LL$ satisfies $N^{G'}[L]\subseteq V(G''_0)$, then Lemma~\ref{lemma-patches}
implies that the subgraph of $G''_0$ drawn in $\Delta_L$ is a subgraph of an $(n,\faces(H,\Sigma(0,0,0)),D,m,a)$-patch.
Let $G'''_0$ be the graph obtained from $G''_0$ by, for each $L\in \LL$ satisfying $N^{G'}[L]\subseteq V(G''_0)$,
removing $L$ and adding edges along the boundary of $\Delta_L$ if they are not already present, so that $G''_0$ is
an $(n,\faces(H,\Sigma(0,0,0)),D,m,a)$-expansion of $G'''_0$.  Note that the maximum degree of $G'''_0$ is at most $n'f_5(n')$. 
Let $T_0=\{s\in S:|\partial s|\le 3, \partial s\subseteq V(G''_0)\}$.

Consider now some $i\in\{1,\ldots, k'\}$.  For each $L\in \LL$ with $N^{G'}[L]\subseteq V(G''_i)$, let
$$L'=L\cup \bigcup_{s\in S, |\partial s|\le 3, \partial s\cap L\neq \emptyset} (V(s)\setminus \partial s).$$
Suppose that $v$ is a vertex of $N^G(L')$, and let $u$ be its neighbor in $L'$.  If $u$ belongs to $V(s)\setminus \partial s$
for some cell $s\in S$ with $\partial s\cap L\neq \emptyset$, then $v$ belongs to $\partial s$,
and since $\partial s$ induces a clique in $G'$, we have $v\in N^{G'}(L)$.  If $u$ belongs to $L$, then
note that $v\not\in V(s)\setminus \partial s$ for every cell $s\in S$, as otherwise we would have $v\in L'$.
Consequently, since $u$ and $v$ are adjacent in $G$, we either have $u,v\in\partial s$ for some cell $s\in S$,
or there exists $j\in \{1,\ldots, k'_0\}$ such that $G_j\subset G''_i$ and $uv\in E(G_j)$.  In both cases, we have $v\in N^{G'}(L)$.  We conclude that $N^G(L')\subseteq N^{G'}(L)$.

Let
\begin{eqnarray*}
\LL'_i&=&\{L':L\in \LL, N^{G'}[L]\subseteq V(G''_i)\}\cup \\
&&\{V(s)\setminus \partial s: |\partial s|\le 3,\partial s\subseteq V(G''_i), (\forall L\in\LL)\,\partial s\cap L=\emptyset\}.
\end{eqnarray*}
Note that for every $K\in \LL'_i$, we have $|N^G(K)|\le f_5(n')$.
Let $G^*_i$ be the graph obtained from $G''_i\cup \bigcup_{s\in S, |\partial s|\le 3, \partial s\subseteq V(G''_i)} s$ by removing edges that do not belong to $G$,
i.e., the edges added in the construction of $T(S)$.  Note that $G^*_i$ is a $p_5$-vortex with $\partial G^*_i=\partial G''_i$.
Let $G^{**}_i$ be the graph obtained from $G^*_i-\bigcup_{K\in\LL'_i} K$
by adding cliques with vertex set $N^G(K)$ for each $K\in \LL'_i$.  For a set $I\subseteq V(G^*_i)$, let us define $I'$
as the set obtained from $I$ by replacing each $v\in I\setminus V(G^{**}_i)$ by the vertices of $N^G(K)$, where $K\in \LL'_i$
is the set satisfying $v\in K$.  Observe that if $I$ separates two vertices of
$\partial G^*_i$ in $G^*_i$, then $I'$ separates them in $G^{**}_i$, and that $|I'|\le f_5(n')|I|$.  Consequently, $G^{**}_i$ is a $p_5f_5(n')$-vortex.
Let $X'_1$, \ldots, $X'_o$ be the path decomposition of $G^{**}_i$ obtained using Lemma~\ref{lemma-pvortex}.
Let $X_1$, \ldots, $X_o$ be the corresponding path decomposition of $G^*_i$ obtained by adding each $K\in \LL'_i$ to a bag
containing $N^G(K)$ (such a bag exists, since $N^G(K)$ induces a clique in $G^{**}_i$).  Let $v_1$, \ldots, $v_o$ be
the vertices of $\partial G^*_i$ in order, so that $v_j\in X_j$ for $1\le j\le o$.  Let $Y_j=\{v_j\}\cup (X_j\cap (X_{j-1}\cup X_{j+1}))$ (where
$X_0=X_{o+1}=\emptyset$) and let $Y=Y_1\cup \ldots\cup Y_o$.  We say that two edges of $G^*_i$ are equivalent if there exists a path joining them without internal
vertices in $Y$.  Let $T_i$ consist of all graphs $t\subseteq G^*_i$ such that $E(t)$ is a maximal set of pairwise equivalent edges,
$V(t)$ is the set of vertices incident with $E(t)$ and $\partial t=V(t)\cap Y$.
Let $G'''_i$ be the graph with the vertex set $Y$ in that two vertices are adjacent iff they belong to $\partial t$ for some $t\in T_i$.
Let us remark that if two vertices $u,v\in Y$ are adjacent in $G^*_i$, then there exists $t\in T_i$ with
$V(t)=\partial t=\{u,v\}$ and $E(t)=\{uv\}$, and thus $uv\in E(G'''_i)$.

Consider a vertex $v\in V(G'''_i)$ that does not belong to $A$, for some $i\ge b$.
Let $uv$ be an edge of $E(G'''_i)$. We assign an edge $e(uv)\in E(G^{**}_i)$ to $uv$ in the
following way.  If $uv\in E(G^{**}_i)$, then we set $e(uv)=uv$.
Otherwise, there exists $t\in T_i$ with $u,v\in \partial t$ and an edge $vw\in E(t)$
with $w\in V(t)\setminus \partial t$.  If $uw\in E(G^{**}_i)$, then we set $e(uv)=uw$.
Otherwise, there exists $K\in\LL'_i$ with $w\in K$.
Let $P\subseteq t$ be a path starting with the edge $vw$ and ending in $u$ and with no internal
vertices in $\partial t$, which exists by the definition of the elements of $T_i$.  Let $w'$ be the vertex of $P$ not belonging to $K\cup\{v\}$ which
is nearest to $v$.  Note that $w'\in N^G(K)$, and thus $vw'$ is an edge of $G^{**}_i$.
Since $uv$ is not an edge of $G^{**}_i$, we have $w'\in V(t)\setminus\partial t$.
We set $e(uv)=vw'$.

Let $zv$ be an edge of $G^{**}_i$.  If $zv\in E(G'''_i)$, then there exists exactly
one edge of $G'''_i$ assigned to $zv$ by $e$, namely the edge $zv$ itself.
If $zv\not\in E(G'''_i)$, then there exists unique $t\in T_i$ with $z\in V(t)\setminus\partial t$,
and if $e(uv)=zv$, then $u$ belongs to $\partial t$.  Therefore, $e(g)=zv$ for at most
$p$ edges $g\in E(G'''_i)$.  Therefore, $\deg_{G'''_i}(v)\le p\deg_{G^{**}_i}(v)\le pn'f_5(n')\le D$,
where the inequality $\deg_{G^{**}_i}(v)\le n'f_5(n')$ follows by Lemma~\ref{lemma-cutbig}.

Let $T^*=T_0\cup\ldots\cup T_{k'}$.  Let $G'''=G''_0\cup G'''_1\cup\ldots\cup G'''_{k'}$.
Let $(T,\beta)$ be the star decomposition of $G$ with center $s$ such that
$\beta(s)=V(G''')$ and with tips $V(t)$ for all $t\in T^*$.
Since $\TT''_0$ is a respectful tangle in $G''_0$ of order greater than $p$ and it is conformal with $\TT$,
for every $t\in T^*$ we have $(t,r)\in \TT$ for the separation $(t,r)$ of $G$ with $V(t\cap r)=\partial t$.
Therefore, $(T,\beta)$ is $\TT$-respecting.  Observe that $\tau(s)=G'''$.
Consequently, $\tau(s)-A$ is an $(n,\faces(H,\Sigma(0,0,0)),D,m,a)$-expansion of
the graph $Q=(G'''_0\cup G'''_1\cup\ldots\cup G'''_{k'})-A$.
Furthermore, all vertices of $Q$ of degree greater than $D$ belong to one of $G'''_1$, \ldots, $G'''_{b-1}$, as required.
\end{proof}

We are now ready to prove a local form of the structure theorem.

\begin{lemma}\label{lemma-local}
For every graph $H$ and integer $\ell$, there exist integers $a$, $D$, $k$, $m$, $n$, $p$ and $\theta$ with the following property.
If $G$ is a graph that does not contain $H$ as a topological minor and $G$ has a tangle $\TT$ of order at least $\theta$ that does
not control a $K_{\ell}$-minor, then 
there exists a $\TT$-respecting star decomposition $(T,\beta)$ of $G$ with center $s$ and a set $A\subseteq \beta(s)$ of size at most $a$
satisfying one of the following:
\begin{itemize}
\item the maximum degree of $\tau(s)-A$ is at most $D$; or,
\item for a surface $\Sigma$ in that $H$ cannot be embedded, $\tau(s)-A$ is an outgrowth
of a graph embedded in $\Sigma$ by at most $k$ vortices of depth at most $p$; or,
\item for a surface $\Sigma$ such that $H$ can be embedded in $\Sigma$
and $\faces(H,\hat{\Sigma})\ge 2$, there exists an outgrowth $Q$ of a graph embedded in $\Sigma$
by at most $k$ vortices of depth at most $p$ such that 
\begin{itemize}
\item all vertices $v\in V(Q)$ with $\deg_{\tau(s)-A}(v)>D$ belong to one of the vortices,
\item less than $\faces(H,\hat{\Sigma})$ of the vortices contain such a vertex, and
\item $\tau(s)-A$ is an $(n,\faces(H,\Sigma(0,0,0)),D,m,a)$-expansion of $Q$.
\end{itemize}
\end{itemize}
\end{lemma}
\begin{proof}
Let $g$ be the smallest integer such that $K_\ell$ can be embedded in every surface of genus at least $g$.
Let $\SSS$ be the set of surfaces of genus less than $g$ in that $H$ can be embedded.
Let $k_0$ be the constant $k$ of Theorem~\ref{thm-gm} for $K_{\ell}$.  
Let $a_1$, $k_1$, $p_1$, $D_1$, $m_1$, $n_1$ and $\phi_1$
be non-decreasing positive functions of one variable such that $a_1(p_0)$, $n_1(p_0)$, $p_1(p_0)$, $D_1(p_0)$, $m_1(p_0)$, $n_1(p_0)$, $n_1(p_0)$ and $\phi_1(p_0)$
are greater or equal to the maximum over $\Sigma\in\SSS$ of constants given by Lemma~\ref{lemma-sublocal} applied for $H$, $k_0$, $\Sigma$ and $p_0$.
If $\SSS=\emptyset$, we let all these functions be constant, assigning to every $p_0$ the value $1$.
Let $\rho_0$, $a_0$ and $\theta_0$ be the constants of Theorem~\ref{thm-gm} for $K_{\ell}$ and the function $\phi_1$.
We set $a=a_0+a_1(\rho_0)$, $D=D_1(\rho_0)$, $k=\max(k_0,k_1(\rho_0))$, $m=m_1(\rho_0)$, $n=n_1(\rho_0)$, $p=\max(2\rho_0+1,p_1(\rho_0))$ and $\theta=\theta_0$.

Since $\TT$ does not control a $K_{\ell}$-minor, we can apply Theorem~\ref{thm-gm} for $K_{\ell}$ and the function $\phi_1$.
We obtain a set $A_0$ of size at most $a_0$ and a $(\TT-A_0)$-central maximal segregation $S$ of $G-A_0$ of type $(p',k_0)$ for some
$p'\le\rho_0$ which has an arrangement in a surface $\Sigma$ of genus less than $g$.
Furthermore, $T(S)$ contains a respectful tangle $\TT'$ of order at least $\phi_1(p')$ conformal with $\TT-A_0$,
and if $f_1$ and $f_2$ are vortex faces of $T(S)$ corresponding to distinct $p'$-vortices and $d'$ is the metric
defined by $\TT'$, then $d'(f_1,f_2)\ge\phi_1(p')$.

If $H$ cannot be embedded in $\Sigma$, then the segregation $S$ can be transformed into a star decomposition satisfying the second outcome of Lemma~\ref{lemma-local}
in the same way as in the proof of (1.3) (assuming (3.1)) in \cite{robertson2003graph}.

If $H$ embeds in $\Sigma$, then we apply Lemma~\ref{lemma-sublocal} with $H$, $k_0$, $p'$ and $\Sigma$ for $G-A_0$ and its segregation $S$ and tangles $\TT-A_0$
and $\TT'$ that we obtained in the previous paragraph. Let $(T,\beta')$ be the resulting $(\TT-A_0)$-respecting star decomposition of $G-A_0$ with center $s$,
and let $A_1\subseteq V(G-A_0)$ be the set such that $\tau(s)-A_1$ satisfies one of the outcomes of Lemma~\ref{lemma-sublocal}.
We let $\beta(v)=\beta'(v)\cup A_0$ for every $v\in V(T)$ and $A=A_0\cup A_1$.  Clearly, $(T,\beta)$ is a $\TT$-respecting star decomposition
of $G$ satisfying either the first or the third outcome of Lemma~\ref{lemma-local}.
\end{proof}

Proving our main result is now easy.

\begin{proof}[Proof of Theorem~\ref{thm-main}]
We follow the proof of Theorem 4.1 of Grohe and Marx~\cite{gmarx}, using Lemma~\ref{lemma-local} instead of
Lemma 6.10 (let us remark that if a tangle controls a minor of $K_{\ell}$, then the corresponding image of $K_{\ell}$ is not removed by the tangle);
the only difference is that $m$ must be chosen as the maximum of $\theta$ from our Lemma~\ref{lemma-local}
and $m^{\star}$ from Lemma~4.10 of \cite{gmarx}.  Let us note that in the derivation of the analogue of Lemma~4.9 of \cite{gmarx},
we move the vertices of $X$ to the apex set of the center of the star decomposition.
\end{proof}

\section{Applications}\label{sec-appl}

A number of special cases of Theorem~\ref{thm-main} is of interest.  The graph $K_5$ is the smallest clique for that the structure of the
graphs that avoid it as a topological minor is not satisfactorily described.  Since $\faces(K_5,\Sigma)=1$ for every surface $\Sigma$
other than the sphere, we obtain the following.

\begin{corollary}\label{cor-nok5}
There exist constants $D$, $k$, $p$, and $a$ with the following property.
Every graph $G$ that does not contain $K_5$ as a topological minor can be expressed as a clique-sum of
graphs $G_1$, $G_2$, \ldots, $G_n$ such that for $1\le i\le n$, there exists a set $A_i\subset V(G_i)$ of size at most $a$
and $G_i-A_i$ either has maximum degree at most $D$, or it is an outgrowth of a graph embedded in the sphere
by at most $k$ vortices of depth at most $p$.
\end{corollary}

It seems plausible that it is actually possible to avoid the vortices and the apex vertices in the embedded case
of Corollary~\ref{cor-nok5}, and that the following holds.

\begin{conjecture}\label{conj-nok5}
There exist constants $D$ and $a$ with the following property.
Every graph $G$ that does not contain $K_5$ as a topological minor can be expressed as a clique-sum of
graphs $G_1$, $G_2$, \ldots, $G_n$ such that for $1\le i\le n$, either $G_i$ is planar or $G_i$ contains at most
$a$ vertices of degree at least $D$.
\end{conjecture}

Let us remark that Seymour (1975) and independently Kelmans (1979) conjectured that
every $5$-connected graph without $K_5$ topological minor is planar.  This, if true, would give some support to Conjecture~\ref{conj-nok5}.

In the rest of this section, we prove Corollary~\ref{cor-adm} and Theorem~\ref{thm-converse} stated in the introduction.
The following construction of a graph with large $\infty$-admissibility was suggested by Patrice Ossona de Mendez.
Let $W_t$ be the graph obtained from $t^2\times t^2$ plane wall by subdividing each of the edges
on its perimeter once and by adding $t$ vertices of degree $t$ attaching to the resulting vertices of degree two in 
a planar way.  Note that $W_t$ is a planar graph and $\faces(W_t,\Sigma)=1$ for every surface $\Sigma$.

\begin{proof}[Proof of Corollary~\ref{cor-adm}.]
We can assume that $t\ge 2$.  Let $v_1$, $v_2$, \ldots, $v_n$ be an ordering
of $V(G)$ with $\infty$-admissibility at most $t$.  Suppose that $G$ contains $H=W_{t+2}$ as a topological minor,
and let $u_1$, $u_2$, \ldots, $u_{t+2}$ be the vertices of $H$ of degree $t+2$.
Let $v_i$ be the last branch vertex of degree $t+2$ of the topological minor of $H$ in $G$, in the ordering of $V(G)$.
Without loss of generality, we can assume that $v_i$ corresponds to $u_{t+2}$.
Note that $H$ contains paths from $u_{t+2}$ to $u_1$, \ldots, $u_{t+1}$ that intersect only in $u_{t+2}$,
and let $P_1$, $P_2$, \ldots, $P_{t+1}$ be the corresponding paths in $G$.  The paths $P_1$, \ldots, $P_{t+1}$ intersect
only in $v_i$ and their other ends are branch vertices of degree $t+2$, and thus they appear before $v_i$ in the ordering.
This shows that $\infty$-backconnectivity of $v_i$ is at least $t+1$, which is a contradiction.

Therefore, $G$ does not contain $W_{t+2}$ as a topological minor, and the corollary follows from Theorem~\ref{thm-main}.
\end{proof}

Let us now show a weak converse of Corollary~\ref{cor-adm}.  We start with two auxiliary claims.
Let $G$ be a graph and let $X$ be a set of vertices of $G$. We say that $G$ has \emph{$X$-based $\infty$-admissibility at most $t$}
if there exists an ordering $v_1$, $v_2$, \ldots, $v_n$ of $V(G)$ with $\infty$-admissibility at most $s$ such that $X=\{v_1,v_2,\ldots, v_{|X|}\}$.
We say that $G$ has \emph{clique-based $\infty$-admissibility at most $t$} if it has $X$-based $\infty$-admissibility at most $t$ for every
$X\subseteq V(G)$ that induces a clique in $G$.

\begin{lemma}\label{lemma-baseadm}
If a graph $G$ has at most $a$ vertices of degree greater than $D$, then $G$ has clique-based $\infty$-admissibility at most $a+D$.
\end{lemma}
\begin{proof}
Suppose that $X\subseteq V(G)$ induces a clique in $G$, and let $A$ be the set of vertices of $V(G)\setminus X$ whose degree in $G$ is greater than $D$.
If $|X|>a-|A|$, then at least one vertex of $X$ has degree at most $D$, and since $X$ induces a clique, we have $|X|\le D+1$.
Therefore, $|X|\le \max(a-|A|,D+1)$, and thus $|A|+|X|\le \max(a,|A|+D+1)\le a+D+1$.  Let $v_1$, \ldots, $v_n$ be an ordering
of $V(G)$ such that $X=\{v_1,\ldots, v_{|X|}\}$ and $A=\{v_{|X|+1},\ldots, v_{|A|+|X|}\}$.  For $1\le i\le |A|+|X|$, the
$\infty$-backconnectivity of $v_i$ is at most $i-1\le |A|+|X|-1\le a+D$.  If $|A|+|X|<i\le n$, then $v_i$ has degree at most $D$,
and thus the $\infty$-backconnectivity of $v_i$ is at most $D$.  Therefore, $G$ has $X$-based $\infty$-admissibility at most $a+D$,
and since the choice of the clique $X$ was arbitrary, it also has clique-based $\infty$-admissibility at most $a+D$.
\end{proof}

\begin{lemma}\label{lemma-treeadm}
Let $G=G_1\cup G_2$, where $G_1\cap G_2$ is a clique.  If both $G_1$ and $G_2$ have clique-based $\infty$-admissibility at most $t$,
then $G$ has clique-based $\infty$-admissibility at most $t$ as well.
\end{lemma}
\begin{proof}
Suppose that $X\subseteq V(G)$ induces a clique.  By symmetry, we can assume that $X\subseteq V(G_1)$.
Let $n_1=|V(G_1)|$ and $n_2=|V(G_2)|$.
Let $v_1$, \ldots, $v_n$ be an ordering of $V(G)$ defined as follows: $v_1$, \ldots, $v_{n_1}$ is an ordering of $V(G_1)$
with $\infty$-admissibility at most $t$ such that $X=\{v_1,\ldots, v_{|X|}\}$.  Let $v_{n_1+1}$, \ldots, $v_n$ be an ordering of $V(G_2)\setminus V(G_1)$
chosen so that there exists an ordering $u_1$, \ldots, $u_{|V(G_1)\cap V(G_2)|}$, $v_{n_1+1}$, \ldots, $v_n$ of $V(G_2)$ with $\infty$-admissibility at most $t$.

Consider a path $P$ in $G$ with ends $v_i$ and $v_j$ such that $j<i$.  If $P$ does not intersect $V(G_1)\cap V(G_2)$, then let $P'=P$.  Otherwise,
let $x$ be the vertex of $P$ nearest to $v_i$ that belongs to $V(G_1)\cap V(G_2)$, and let $y$ be the vertex of $P$ nearest to $v_j$ that belongs to $V(G_1)\cap V(G_2)$.
Let us remark that we can have $x=v_i$ or $y=v_j$.
If $i\le n_1$, then let $P'$ be the path consisting of the subpath of $P$ between $v_i$ and $x$, the edge $xy$, and the subpath of $P$ between $y$ and $v_j$.
If $i>n_1$, then let $P'$ be the subpath of $P$ between $v_i$ and $x$.  Observe that $V(P')\subseteq V(P)$, the end of $P'$ distinct from $v_i$ appears before $v_i$
in the ordering, and either $P'\subseteq G_1$ or $P'\subseteq G_2$.  Consequently, the $\infty$-backconnectivity of $v_i$ in the ordering of vertices of $G$ is at most
the $\infty$-backconnectivity of $v_i$ in the ordering of vertices of $G_j$, where $j=1$ if $i\le n_1$ and $j=2$ if $i>n_1$.  In both cases, the $\infty$-backconnectivity of $v_i$
is at most $t$.

It follows that the ordering of vertices of $G$ has $\infty$-backconnectivity at most $t$, and since the choice of the clique $X$ was arbitrary,
$G$ has clique-based $\infty$-admissibility at most $t$.
\end{proof}

\begin{proof}[Proof of Theorem~\ref{thm-converse}.]
We prove that the claim holds with $T=a+D$.  Let $G$ be a clique-sum of graphs $G_1$, \ldots, $G_n$ with at most $a$ vertices
of degree greater than $D$.  We can choose the ordering of the graphs so that $G_1\cup \ldots \cup G_i$ intersects $G_{i+1}$ in
a clique, for $1\le i\le n-1$.  By Lemmas~\ref{lemma-baseadm} and \ref{lemma-treeadm}, we conclude that $G'=G_1\cup \ldots\cup G_n$
has clique-based $\infty$-admissibility at most $T$.  Since $G$ is a subgraph of $G'$, it follows that $G$ has $\infty$-admissibility at most $T$.
\end{proof}

\section{Acknowledgments}
I would like to thank Patrice Ossona de Mendez and Paul Wollan for discussions related to
the $\infty$-admissibility and immersion problems.

\bibliographystyle{siam}
\bibliography{topmin}

\end{document}